%% file: arxiv_intr_mfem_cosserat.tex
\documentclass[11pt]{amsart}

\usepackage{amssymb}	
\usepackage{tikz}
\usepackage{tikz-cd}
\usepackage{pgfplots}

\usepackage{amsmath}
\usepackage{mathrsfs}
\usepackage{amsthm}
\usepackage{verbatim}
\usepackage{subeqnarray}
\usepackage{graphicx}
\usepackage{cancel}
\usepackage{color}
\usepackage{float}
\usepackage{bm}
\usepackage{hyperref}
\usepackage{amsfonts}
\usepackage{amscd}
\usepackage{cases}
\usepackage{xcolor}
\usepackage{eucal}
\usepackage[nice]{nicefrac}

\usepackage{ifthen}

\pgfplotsset{compat=1.18}
\usetikzlibrary{calc}
\usetikzlibrary{spy}
\def\centerarc[#1](#2)(#3:#4:#5)
{ \draw[#1] ($(#2)+({#5*cos(#3)},{#5*sin(#3)})$) arc (#3:#4:#5); }

\definecolor{red}{rgb}{0.8,0,0}
\definecolor{darkorange}{rgb}{1,0.4,0}
\definecolor{lightorange}{rgb}{1,0.6, 0}
\definecolor{yellow}{rgb}{1,0.8, 0}

\newtheorem{theorem}{Theorem}
\newtheorem{remark}{Remark}

\newtheorem{definition}{Definition}
\newtheorem{corollary}{Corollary}
\newtheorem{lemma}{Lemma}
\newtheorem{problem}{Problem}

\usepackage{amsopn}

\DeclareMathOperator*{\argmin}{arg\,min}
\newcommand{\bs}{{\scriptscriptstyle \bullet}}

\newcommand{\0}{\mathaccent23}
\newcommand\tr{\operatorname{tr}}
\newcommand\inc{\operatorname{inc}}
\newcommand\skw{\operatorname{skw}}
\newcommand\vskw{\operatorname{vskw}}
\newcommand\mskw{\operatorname{mskw}}

\newcommand\sym{\operatorname{sym}}
\newcommand\grad{\nabla}

\renewcommand\div{\operatorname{div}}

\newcommand\curl{\operatorname{curl}}
\newcommand\dev{\mathrm{dev}}
\newcommand\ran{\mathcal{R}}
\newcommand\M{\mathbb{R}^{3\times 3}}
\renewcommand\S{\mathbb{R}^{3\times 3}_{\sym}}
\newcommand\R{\mathbb{R}}

\newcommand{\couplestress}{m}
\newcommand{\Lag}{\mathrm{Lag}}
\newcommand{\vLag}{\mathrm{\textbf{Lag}}}
\newcommand{\HHJ}{\mathrm{HHJ}}
\newcommand{\MCS}{\mathrm{MCS}}
\newcommand{\BDM}{\mathrm{BDM}}
\newcommand{\RT}{\mathrm{RT}}
\newcommand{\NedII}{\mathrm{Ned}_{II}}
\newcommand{\IntRT}[1][]{\mathcal{I}^{\RT\ifthenelse{\equal{#1}{}}{}{^{#1}}}_h}

\newcommand{\half}{\nicefrac{1}{2}}

\usepackage{chngpage}

\usepackage{lipsum}

\usepackage{booktabs}

\makeatletter
\@addtoreset{equation}{section}
\makeatother

\usepackage{geometry}
\title[Intrinsic mixed FEM for Cosserat elasticity]{Intrinsic mixed finite element methods for linear Cosserat elasticity}
\author[A.~Dziubek]{Andrea~Dziubek}
\address{SUNY Polytechnic Institute, Utica NY 13502, USA}
\email{dziubea@sunypoly.edu}
\author[K.~Hu]{Kaibo~Hu}
\address{The University of Edinburgh, James Clerk Maxwell Building, Peter Guthrie Tait Rd, Edinburgh EH9 3FD, UK}
\email{kaibo.hu@ed.ac.uk}
\author[M.~Karow]{Michael~Karow}
\address{Technische Universit\"at Berlin, Berlin, DE}
\email{karow@math.tu-berlin.de}
\author[M.~Neunteufel]{Michael~Neunteufel}
\address{Portland State University, PO Box 751, Portland OR 97201, USA}
\email{mneunteu@pdx.edu}
\begin{document}

\date{\today}

\begin{abstract}
    We propose two parameter-robust mixed finite element methods for linear Cosserat elasticity. The Cosserat coupling constant $\mu_c$, connecting the displacement $u$ and rotation vector $\omega$, leads to possible locking phenomena in finite element methods. The formal limit of $\mu_c\to\infty$ enforces the constraint $\frac{1}{2}\operatorname{curl} u = \omega$ and leads to the fourth-order couple stress problem. Viewing the linear Cosserat model as the Hodge-Laplacian problem of a twisted de~Rham complex, we derive structure-preserving distributional finite element spaces, where the limit constraint is fulfilled in the discrete setting. Applying the mass conserving mixed stress (MCS) method for the rotations, the resulting scheme is robust in $\mu_c$. Combining it with the tangential-displacement normal-normal-stress (TDNNS) method for the displacement part, we obtain additional robustness in the nearly incompressible regime and for anisotropic structures. Using a post-processing scheme for the rotations, we prove optimal convergence rates independent of the Cosserat coupling constant $\mu_c$.  
  We demonstrate the performance of the proposed methods in several numerical benchmark examples.\\
	\vspace*{0.25cm}
	
\end{abstract}

\maketitle

\section{Introduction}
\label{sec:introduction}
In 1909, the Cosserat brothers \cite{Cosserat1909} introduced the idea of modeling a material as a collection of points with associated directions that can rotate and stretch independently of the displacement of material points. G{\"u}nther \cite{Gunther1958de} derived Cosserat equations for linear deformations. A comprehensive review of recent advances, including literature on the history of generalized continua, plates, and shells, can be found in \cite{AMM2010}. The Cosserat models have been proven to be effective for various materials with complex microstructures, including micropolar fluids and electromagnetic solids (cf. \cite{CISMcosserat2010}) and man-made Cosserat materials \cite{RuegerLakes2018}. Identifying the additional material parameters for the microstructure remains challenging  (despite Lakes' pioneering size-effect experiments, cf. \cite{Lakes2023}). However, the number of the material parameters could be reduced through mathematical analysis \cite{NJF2010,NJ2009}.

There have been extensive studies on the classical elasticity models with finite element methods. An important numerical issue is locking - the convergence may deteriorate, or the solution tends to be zero as some physical parameters approach some limit values (an example is when the Poisson ratio $\nu$ goes to $1/2$, or equivalently the Lam\'e constant $\lambda\to\infty$). Various numerical schemes and finite elements have been proposed to avoid locking. In contrast, numerical schemes for Cosserat models have received less attention, with exceptions such as \cite{riahi2009full,sander2010geodesic,providas2002finite,grbvcic2018variational,XHZH2019}.  Several finite elements have been proposed for two- and three-dimensional linear Cosserat elasticity and the coupled stress problem, which can be obtained as the limit of the Cosserat model for the Cosserat coupling constant $\mu_c\to\infty$. A hybrid-stress element to avoid locking problems has been presented in \cite{XHZH2019}. 
Cosserat models involve more physical parameters, including a characteristic length scale and three micropolar moduli. Nevertheless, to the best of our knowledge, a method that is robust with all the parameters for Cosserat and micropolar models and a rigorous numerical analysis is still open. We mention that a recent work \cite{boon2024mixedfiniteelementmethods} constructed a mixed finite element scheme with a rigorous analysis. The method achieves robustness with the limit to classical elasticity when the characteristic length goes to zero. 
However, the limit to the couple stress problem ($\mu_{c}\rightarrow \infty$) was not addressed. 

In this paper, we construct schemes for the linear Cosserat problem inspired by Finite Element Exterior Calculus (FEEC) \cite{arnold2018finite,Arnold.D;Falk.R;Winther.R.2006a,Arnold.D;Falk.R;Winther.R.2010a}. A key idea of FEEC is to discretize the underlying differential complexes of PDEs.   The linear Cosserat model can be viewed as the Hodge-Laplacian problem of the elasticity version of the twisted de~Rham complex at index zero \cite{arnold2021complexes,vcap2022bgg}. These complexes are intermediate steps for deriving the BGG complexes (e.g., the elasticity complex) and incorporating more information. In the elasticity case, the twisted complex encodes displacement and pointwise rotation, while the elasticity complex (encoding linear elasticity) is obtained by eliminating the variables corresponding to rotation.

Fitting conforming finite elements in the twisted complexes would require higher order smoothness (cf. diagram \eqref{elasticity-diagram}). 
Instead of pursuing smooth finite elements, we follow the trend of {\it distributional finite elements}. The idea is to impose weaker regularity and allow Dirac deltas as shape functions. The distributional finite element de~Rham complex was initially investigated by Braess and Sch\"oberl \cite{braess2008equilibrated} for equilibrated residual error estimators. Dirac deltas can be paired with fields of proper continuity. This led to the tangential-displacement-normal-normal-stress (TDNNS) method for linear elasticity \cite{pechstein2011tangential}. Then the idea has also been applied to Reissner--Mindlin plates \cite{PS17}, the mass conserving mixed stress (MCS) method for Stokes equations \cite{gopalakrishnan2020mass0}, and the Hellan--Herrmann--Johnson method for the Kirchhoff--Love plate equation \cite{Hel67,Her67,Joh73} and shells \cite{NS2019,NS2024}. An attractive feature of distributional elements is that they usually have invariant forms under maps (pullbacks), which is important for finite element computation and generalizations to manifolds. The (affine) invariance reflects certain coordinate independence. For example, in Regge calculus and its finite element interpretation \cite{christiansen2011linearization,regge1961general}, metrics are discretized by the edge lengths, which are coordinate-independent. Therefore, these finite elements are intrinsic. We refer to \cite{hu2023distributional,christiansen2023extended} for constructing distributional finite elements for BGG diagrams and a discrete exterior calculus interpretation.  

In this paper, we propose two robust mixed finite element methods based on twisted complexes, called the MCS method (Problem \ref{prob:mcs_method} below) and the TDNNS-MCS method (Problem \ref{prob:mcs_tdnns_method} below), respectively. Using the MCS discretization method for the linearized rotations with 
corresponding couple stress tensor circumvents locking with respect to the Cosserat coupling constant $\mu_c$. When using the TDNNS method for the elasticity part on top of the MCS method, additional robustness for anisotropic structures and nearly incompressible materials is obtained. 
A rigorous error analysis is performed to prove convergence rates independent of the Cosserat coupling constant $\mu_c$ for the proposed mixed formulations. A post-processing procedure for the rotation field is presented. Several numerical benchmark examples demonstrate the proposed methods' efficiency and robustness.

The rest of the paper is organized as follows. In Section~\ref{sec:cosserat_couplestress}, we present the equations of linear Cosserat elasticity
and their relation with twisted complexes. In Section~\ref{sec:function_spaces}, we derive the two finite element methods. In Section~\ref{sec:stability_error} we perform a rigorous error analysis of the proposed methods with respect to the Cosserat coupling constant $\mu_c$. Further, post-processing schemes for the rotation field are discussed. We present several numerical benchmark examples in Section~\ref{sec:numerics} to show the performance and robustness of our methods.

\section{Linear Cosserat elasticity}
\label{sec:cosserat_couplestress}
Let $\Omega\subset \R^3$ be a domain. Define the following algebraic and differential operators for (differentiable) vectors $u$ and matrices $A$
\begin{flalign*}
	&\grad u = \begin{pmatrix}
		\frac{\partial u_x}{\partial x} & \frac{\partial u_x}{\partial y} & \frac{\partial u_x}{\partial z}\\
	 	\frac{\partial u_y}{\partial x} & \frac{\partial u_y}{\partial y} & \frac{\partial u_y}{\partial z}\\
	 	\frac{\partial u_z}{\partial x} & \frac{\partial u_z}{\partial y} & \frac{\partial u_z}{\partial z}
	\end{pmatrix}  ,\qquad \div u = \nabla\cdot u,\qquad \curl u = \nabla\times u,\\
	 & \mskw u = \begin{pmatrix}
		0 & -u_z & u_y\\
		u_z & 0 & -u_x\\
		-u_y & u_x & 0
	\end{pmatrix},\qquad \vskw A = \frac{1}{2}\begin{pmatrix}
		A_{zy}-A_{yz}\\
		A_{xz}-A_{zx}\\
		A_{yx}-A_{xy}
	\end{pmatrix},
\end{flalign*}
and denote by $\tr$, $\dev$, $\sym$, and $\skw$ the trace, deviatoric, symmetric, and skew part of a matrix, respectively. The identity matrix is denoted by $I$. 

\subsection{Cosserat energy and strong form}
\label{subsec:cosserat-energy}

We consider the following material laws
\begin{subequations}
\label{eq:material_law}
\begin{flalign}
\label{eq:material_law_elasticity}
	C_1(\varepsilon)&= 2\mu\sym\varepsilon+\lambda\tr\varepsilon\,I+\mu_c\skw\varepsilon=\mathcal{C}(\varepsilon)+\mu_c\skw\varepsilon,\\
		C_2(\varepsilon)&= (\gamma+\beta)\sym\varepsilon + \alpha \tr\varepsilon\,I+(\gamma-\beta)\skw\varepsilon\label{eq:material_law_curvature},
\end{flalign}
\end{subequations}
where $\mathcal{C}$ is the classical elasticity tensor with \emph{Lam\'e parameters} $\mu>0$ and $\lambda\geq 0$, $\mu_c\geq 0$ is the \emph{Cosserat coupling} constant, and $\alpha,\beta,\gamma\in\R$ are additional so-called \emph{micropolar moduli}. In the linear Cosserat setting, we consider the displacement field $u:\Omega\to\R^3$ and the rotation matrix identified with its axial rotation vector $\omega:\Omega\to\R^3$. 
The \emph{Cosserat energy functional} (with body forces $f_u$ and $f_{\omega}$) is defined by, see e.g. \cite{NJ2009},
\begin{flalign}
		\mathcal{E}^{\mathrm{Cosserat}}(u,\omega)&:=\int_{\Omega}\Big(\frac{1}{2}\|\grad u -\mskw \omega\|_{C_{1}}^{2}+\frac{1}{2}\|\grad\omega\|_{C_{2}}^{2}-\langle f_u,u\rangle - \langle f_\omega,\omega\rangle\Big)\,dx\nonumber\\
		&=\int_{\Omega}\Big(\frac{1}{2}\|\sym\grad u\|^{2}_{\mathcal{C}}+\mu_{c}\left\|\half\curl u-\omega \right\|^{2}+ \frac{\gamma+\beta}{2}\|\sym \grad\omega\|^{2}\label{eq:cosserat-energy}\\
		&\qquad\quad +\frac{\gamma-\beta}{4}\|\curl \omega\|^{2}+\frac{\alpha}{2}\|\div \omega\|^{2}\Big)\,dx-\int_{\Omega}(\langle f_u,u\rangle + \langle f_\omega,\omega\rangle)\,dx.\nonumber
\end{flalign}
The, in general non-symmetric, strain field $\varepsilon=\grad u - \mskw\omega$ is called the \emph{first Cosserat stretch} tensor and $\kappa=\grad \omega$ the \emph{micropolar curvature} tensor. Note that \eqref{eq:cosserat-energy} yields a minimization problem.

\begin{remark}[Linear elasticity limit]
	\label{rem:linear_elasticity_limit}
	When setting $\mu_c=0$ in \eqref{eq:cosserat-energy}, the Cosserat problem decouples into the standard linear elasticity problem in $u$ and an independent vector-Poisson problem in $\omega$. Considering the limit $C_2\to 0$ is also possible with $\mu_c>0$. Then, $\omega=\half\curl u$ absorbs the skew-symmetric part of $\grad u$, yielding linear elasticity. When discretized in a Hellinger--Reissner formulation that includes a nonsymmetric stress, $\omega$ can be interpreted as a Lagrange multiplier enforcing the weak symmetry of the stress, see \cite{arnold2007mixed,boon2024mixedfiniteelementmethods}. For ease of presentation, we assume that $C_2$ is positive definite and emphasize that it can be relaxed to be positive semi-definite. For a discussion of common choices of $\alpha$, $\beta$, and $\gamma$ we refer to \cite{JN2010}.
\end{remark}

Let $\Omega\subset \R^3$ be a bounded Lipschitz domain. We split the boundary $\Gamma=\partial \Omega$ into a Dirichlet $\Gamma_D$ and Neumann $\Gamma_N$ part such that $\Gamma_D\cup\Gamma_N=\Gamma$ and $\Gamma_D\cap\Gamma_N=\emptyset$, and we denote the outer unit normal by $n$. Assume that Dirichlet data $u_D$ and $\omega_D$ are prescribed on $\Gamma_D$ and that surface traction forces 
are given on $\Gamma_N$ by $g_u$ and $g_{\omega}$. For simplicity, we assume throughout the paper that the Dirichlet parts for $u$ and $\omega$ coincide and emphasize that the extension to distinguished boundary parts for the displacement and rotations is straightforward. By taking variations of \eqref{eq:cosserat-energy} with respect to $u$ and $\omega$ and integration by parts, we readily obtain the strong form with corresponding boundary conditions of Cosserat elasticity. 
By defining the \emph{elasticity stress tensor} $\sigma:=\mathcal{C}(\sym\grad u)$ and the \emph{couple stress tensor} $\couplestress:=C_2(\grad \omega)$ we have
\begin{subequations}
	\label{eq:strong_cosserat}
	\begin{flalign}
		-\div(C_1(\grad u - \mskw\omega))=-\div(\sigma)+ \mu_c\curl\left(\half\curl u-\omega\right)= f_u\quad& \text{ in }\Omega,\label{eq:strong_form_lin_mom}\\
	 -\div(\couplestress)-2\mu_c\left(\half\curl u -\omega\right) = f_{\omega}\quad& \text{ in }\Omega,\label{eq:strong_form_angular_mom}\\
	u=u_D,\quad \omega=\omega_D\quad& \text{ on }\Gamma_D,\\
	  \sigma_n + \mu_c\left(\half\curl u-\omega\right)\times n= g_u,\quad \couplestress_n=g_{\omega}\quad& \text{ on }\Gamma_N,
	\end{flalign}
\end{subequations}
where $f_u$ and $f_{\omega}$ are body forces acting on $u$ and $\omega$, respectively. Further, we use the notation $\sigma_n:=\sigma n$ for the normal component of a tensor. Equation \eqref{eq:strong_form_lin_mom} is the balance of linear momentum, and \eqref{eq:strong_form_angular_mom} is the balance of angular momentum.

\subsection{Linear Cosserat elasticity as twisted de Rham complex}
\label{subsec:cosserat-twisted}
Cosserat energy \eqref{eq:cosserat-energy} corresponds to the \emph{Hodge-Laplacian} problem of the elasticity version of the \emph{twisted de Rham complex} at index zero \cite{vcap2022bgg}. This observation motivates our choices of finite element spaces. In this section, we explain this claim. The \emph{Bernstein-Gelfand-Gelfand} (BGG) construction in \cite{arnold2021complexes,vcap2022bgg} starts with the following diagram
\begin{equation}\label{elasticity-diagram}
	\begin{tikzcd}
	0 \!\arrow{r}\!&\!\left [H^{q}\right]^{3} \!\arrow{r}{\grad}  \!&\!\left[H^{q-1}\right]^{3\times 3}   \!\arrow{r}{\curl}\! &\!\left[H^{q-2}\right]^{3\times 3} \!\arrow{r}{\div} \! & \!\left[H^{q-3}\right]^{3}\!\arrow{r}{} \!& \!0\\
	0 \!\arrow{r}\! &\!\left[H^{q-1}\right]^{3} \!\arrow{r}{\grad}\! \arrow[ur, "\mskw"]\!&\!\left[H^{q-2}\right]^{3\times 3}  \! \arrow{r}{\curl} \arrow[ur, "-\mathcal{S}"]\!&\!\left[H^{q-3}\right]^{3\times 3}  \!\arrow{r}{\div}\arrow[ur, "-2\vskw"]\! &\! \left[H^{q-4}\right]^{3}  \!\arrow{r}{} \!& \!0.
	 \end{tikzcd}
	\end{equation}
Here $\left [H^{q}\right]^{3}$ denotes a vector-valued function space, for which each component is in the Sobolev space $H^{	q}$, and $\left[H^{q-1}\right]^{3\times 3} $ denotes the 3-by-3 matrix version, etc. The operator $\mathcal{S}$ is defined by $\mathcal{S}u=u^{T}-\tr(u)I$, where $u$ is a matrix and $u^T$ denotes the transpose of $u$. 
Both rows of \eqref{elasticity-diagram} are vector-valued de~Rham complexes in 3D, which are connected by algebraic operators.  One may derive a {\it twisted de~Rham complex} (or, twisted complex, for short) and a {\it BGG complex} (the elasticity complex) from \eqref{elasticity-diagram}. Each space $Y^{\bs}$ in the twisted complex consists of two components, one from the first row and one from the second, and the operators $d_{V}^{\bs}$ have the general form 
$$
d_{V}^{k}:=\left (
\begin{array}{cc}
d^{k} & -S^{k} \\
0 & d^{k}
\end{array}
\right ),
$$
where $S^{k}$ is the $k$-th diagonal operator, and $d^k$ denotes the exterior derivative translating into the differential operators in \eqref{elasticity-diagram}. We denote the twisted complex by
\begin{equation}\label{A-sequence-0}
	\begin{tikzcd}
0\arrow{r}& Y^{0} \arrow{r}{d_{V}^{0}} &Y^{1} \arrow{r}{d_{V}^{1}} &Y^{2} \arrow{r}{d_{V}^{2}} &Y^{3} \arrow{r}{} & 0,
	 \end{tikzcd}
	\end{equation}
	where $Y^{0}=\left [H^{q}\right]^{3}\times \left[H^{q-1}\right]^{3}$, $Y^{1}=\left[H^{q-1}\right]^{3\times 3}  \times \left[H^{q-2}\right]^{3\times 3}  $, $Y^{2}=\left[H^{q-2}\right]^{3\times 3}  \times \left[H^{q-3}\right]^{3\times 3}  $, and 
	$Y^{3}=\left [H^{q-3}\right]^{3}\times \left[H^{q-4}\right]^{3}$, and the operators are
	$$
d_{V}^{0}:=\left (
\begin{array}{cc}
\grad & -\mskw\\
0 & \grad
\end{array}
\right ), \quad d_{V}^{1}:=\left (
\begin{array}{cc}
\curl & \mathcal{S}\\
0 & \curl
\end{array}
\right ), \quad d_{V}^{2}:=\left (
\begin{array}{cc}
\div & -2\vskw\\
0 & \div
\end{array}
\right ).
	$$
The sequence \eqref{A-sequence-0} is a complex since $d_{V}^{k}\circ d_{V}^{k-1}=0$ for all $k$ due to the anti-commutativity condition $dS=-Sd$, which holds for the operators in \eqref{elasticity-diagram}. The name ``twisted complex'' comes from the observation that the operators in \eqref{A-sequence-0} have an off-diagonal part $-S$, which twists the two de~Rham complexes. The BGG machinery aims to eliminate spaces in \eqref{elasticity-diagram} connected by the connecting maps as much as possible, leading to the following BGG complex consisting of the kernels and cokernels of the connecting maps:
\begin{equation}\label{sequence:hs}
	\begin{tikzcd}
	0\arrow{r} & \left [H^{q}\right]^{3} \arrow{r}{{\sym\grad}} &\left [H^{q-1}\right]_{\sym}^{3\times 3} \arrow{r}{\inc} &\left [H^{q-3}\right]_{\sym}^{3\times 3} \arrow{r}{\div} &\left [H^{q-4}\right]^{3} \arrow{r} & 0,
	\end{tikzcd}
\end{equation}
where $\inc :=\curl\circ\, \mathcal{S}^{-1}\circ \curl$ is the incompatibility operator 
and $\left[H^{q-1}\right]_{\sym}^{3\times 3}$ denotes the space of symmetric matrices where each entry is a function in the Sobolev space $H^{q-1}$. 
This is the most convenient functional analytic setting for deriving the elasticity (BGG) complex. However, to get well-posed formulations of the Hodge-Laplacian problem, we have 
many other options in the choices of spaces and boundary conditions.  
For the Cosserat problem, we will use the following variants of \eqref{A-sequence-0}:
\begin{equation}\label{elasticity-diagram-Hd}
	\begin{tikzcd}
	0 \!\arrow{r}&\!\left [H^{1}\right]^{3} \!\arrow{r}{\grad}  &\![H(\curl)]^3   \arrow{r}{\curl}  \!&\![H(\div)]^3  \!\arrow{r}{\div}  \!&\left [L^{2}\right]^{3}\! \arrow{r}{} & \!0\\
	0 \!\arrow{r} &\!\left [H^{1}\right]^{3} \!\arrow{r}{\grad} \arrow[ur, "\mskw"]\!& \![H(\curl)]^3  \!\arrow{r}{\curl} \arrow[ur, "-\mathcal{S}"]\!&\![H(\div)]^3  \!\arrow{r}{\div}\arrow[ur, "2\vskw"] \!& \!\left [L^{2}\right]^{3} \!  \arrow{r}{} & \!0,
	 \end{tikzcd}
	\end{equation} 
or the version with homogeneous Dirichlet boundary conditions:
\begin{equation}\label{elasticity-diagram-Hd0}
	\begin{tikzcd}
		0 \!\arrow{r}\!&\!\left [H_{0}^{1}\right]^{3} \!\arrow{r}{\grad}  \!&\![H_{0}(\curl)]^3  \! \arrow{r}{\curl}  &\![H_{0}(\div)]^3  \!\arrow{r}{\div} \!&\!\left [L_{0}^{2}\right]^{3} \!\arrow{r}{} \!&\! 0\\
		0 \!\arrow{r} \!&\!\left [H_{0}^{1}\right]^{3} \!\arrow{r}{\grad} \arrow[ur, "\mskw"]\!&\![H_{0}(\curl)]^3  \!\arrow{r}{\curl} \arrow[ur, "-\mathcal{S}"]\!&\![H_{0}(\div)]^3 \! \arrow{r}{\div}\arrow[ur, "2\vskw"] \!& \!\left [L_{0}^{2}\right]^{3}  \! \arrow{r}{} \!&\!0.
	 \end{tikzcd}
	\end{equation} 

Above, we used the vector-valued function spaces 
\begin{flalign*}
	&H(\curl)\!:=\!\{u\!\in \![L^{2}(\Omega)]^3: \curl u \in [L^{2}(\Omega)]^3\},\, H(\div)\!:=\!\{u\!\in\! [L^{2}(\Omega)]^3: \div u \in L^{2}(\Omega)\},\\
	&H_0(\curl)\!:=\!\{u\!\in\! H(\curl): \tr_t u \!=\!0 \text{ on }\partial\Omega\},\,\, H_0(\div)\!:=\!\{u\!\in\! H(\div): \tr_n u \!=\!0 \text{ on }\partial\Omega\},
\end{flalign*}
where $\tr_t$ and $\tr_n$ denote, respectively, the tangential and normal trace operators, which for continuous $u$ reads $\tr_t u = u\times n$ and $\tr_n u = \langle u, n\rangle$. In \eqref{elasticity-diagram-Hd}--\eqref{elasticity-diagram-Hd0}, the algebraic operators $-\mskw$, $\mathcal{S}$, and $2\vskw$ are indeed maps between the spaces in the diagrams because these maps (anti)commute with the differential operators $\grad$, $\curl$, and $\div$.  
The same argument holds for partial boundary conditions when we impose Dirichlet conditions only on $\Gamma_{D}$. 
In this case, we need a slight modification in \eqref{elasticity-diagram-Hd0}: the last 
space is $\left[L_{0}^{2}\right]^{3}$ (zero mean) if the entire boundary is Dirichlet, i.e., $\Gamma_{D}=\partial \Omega$. Otherwise we use $\left[L^{2}\right]^{3}$.

\begin{remark}
The derivation of the elasticity complex from the twisted complex is based on the following diagram (see, e.g., \cite[Section 11.2]{Arnold.D;Falk.R;Winther.R.2006a} and \cite{arnold2021complexes}):
{ \begin{equation*}\label{cplx:twisted}
  \begin{tikzcd}[ampersand replacement=\&, column sep=2.0cm]
0\arrow{r}\&
\left ( \begin{array}{c}
C^{\infty}\otimes \mathbb{R}^{3} \\
C^{\infty}\otimes \mathbb{R}^{3}
 \end{array}\right )\arrow{d}{\pi^{0}}
 \arrow{r}{
 \begin{pmatrix} \grad & -\mskw \\ 0 & \grad \end{pmatrix}
 }\&  \left ( \begin{array}{c}
C^{\infty}\otimes \mathbb{R}^{3\times 3}  \\
C^{\infty}\otimes \mathbb{R}^{3\times 3} 
 \end{array}\right )\arrow{d}{\pi^{1}} \arrow{r}{
 \begin{pmatrix} \curl & S \\ 0 & \curl \end{pmatrix}
 } \&\cdots
 \\0 \arrow{r} \& \Gamma^{0}\arrow{r}{ \begin{pmatrix} \grad & -\mskw \\ 0 & \grad \end{pmatrix}}\arrow{d}{\cong}\& \Gamma^{1}\arrow{d}{\cong}\arrow{r}{\begin{pmatrix} \curl & S \\ 0 & \curl \end{pmatrix}}\&\cdots
  \\0 \arrow{r} \& C^{\infty}\otimes \mathbb{R}^{3}\arrow{r}{\sym\grad}\& C^{\infty}\otimes \mathbb{R}^{3\times 3}_{\sym}\arrow{r}\&\cdots
\end{tikzcd}
\end{equation*}}
Here, the first line is the twisted complex. In the second line, we have 
\begin{flalign*}
&\Gamma^{0}:=\{(u, \omega)\in C^{\infty}\otimes \mathbb{R}^{3} \times C^{\infty}\otimes \mathbb{R}^{3} : [\grad u-\mskw \omega]\perp \ran (\mskw)\},\\
&\Gamma^{1}:=\{(\sigma, \tau)\in C^{\infty}\otimes \mathbb{R}^{3\times 3} \times C^{\infty}\otimes \mathbb{R}^{3\times 3} : [\curl \sigma+S \tau]\perp \ran (S)\}.
\end{flalign*}
The above condition for $\Gamma^{0}$ implies that, if $(u, \omega)\in \Gamma^{0}$, then $\omega$ is determined by $u$ via  
$\omega=\half\curl u$
(such that $\grad u-\mskw \omega$ is orthogonal to all skew-symmetric matrices); the condition for $\Gamma^{1}$ implies that, if $(\sigma, \tau)\in \Gamma^{1}$, then $\tau$ is determined by $\sigma$ as $\tau=-S^{-1}\curl \sigma$. This complex with $\Gamma^{0}$ and $\Gamma^{1}$ is a subcomplex of the twisted complex -- the spaces are subspaces, while the operators are the same. The idea of this reduction from the twisted complex to the $\Gamma$ complex is that the first component determines the second component in each space. 
The $\Gamma$ complex is further isomorphic to the elasticity complex (the last line in the diagram). 

In summary, standard elasticity (with $u$) can be viewed as a special case of the Cosserat model where the displacement determines the rotation through $\omega=\half\curl u$. 
Another way to derive standard elasticity model from the Cosserat model is to set the parameter $\mu_{c}=0$, cf. Remark~\ref{rem:linear_elasticity_limit}.
\end{remark}

The following observation relates the twisted complex with Cosserat elasticity.
\begin{lemma}	
\label{lem:relation_Cosserat_twisted}
Define the inner product $(\cdot,\cdot)_C$ on $Y^{1}$ by the relation
	$$
	(d_{V}^{0}(u, \omega), d_{V}^{0}(u, \omega))_{C} = \frac{1}{2}\int_{\Omega}(\|\grad u -\mskw \omega\|_{C_{1}}^{2}+\|\grad\omega\|_{C_{2}}^{2})\,dx. 
	$$
	Then the Cosserat energy \eqref{eq:cosserat-energy} is the Hodge-Laplacian 
	\begin{flalign*}
		(d_{V}^{0}(u, \omega), d_{V}^{0}(u, \omega))_{C} -\int_{\Omega}(\langle f_u,u\rangle + \langle f_\omega,\omega\rangle)\,dx  = \mathcal{E}^{\mathrm{Cosserat}}(u,\omega).
	\end{flalign*}
\end{lemma}

\subsection{Primal formulation for Cosserat elasticity}
\label{subsec:cosserat-formulations}

In this section, we derive and discuss the primal formulation for Cosserat elasticity and its discretization.

\subsubsection{Primal formulation}
\label{subsubsec:cosserat-primal}
Let $H^1_{\Gamma}(\Omega)=\{u\in H^1(\Omega)\,:\, \tr u=0 \text{ on }\Gamma\}$ be the set of $H^1$ functions, with zero boundary conditions on $\Gamma\subset\partial\Omega$ and $H^{\half}(\Gamma)$ the trace space of $H^1_{\Gamma}(\Omega)$. We denote with $H^{-1}(\Omega)$ and $H^{-\half}(\Gamma)$ the dual spaces of $H^1_{\Gamma}(\Omega)$ and $H^{\half}(\Gamma)$, respectively. The \emph{primal variation formulation} of \eqref{eq:strong_cosserat}, cf. \eqref{eq:cosserat-energy} and Lemma~\ref{lem:relation_Cosserat_twisted}, reads as follows.
\begin{problem}[Primal formulation]
	\label{prob:primal_cosserat}
	Find $(u,\omega)\in [H^1(\Omega)]^3\times [H^1(\Omega)]^3$ such that $u=u_D$ and $\omega=\omega_D$ on $\Gamma_D$ and for all $(v,\xi)\in [H^1_{\Gamma_D}(\Omega)]^3\times [H^1_{\Gamma_D}(\Omega)]^3$ there holds
	\begin{flalign}
		\label{eq:variation_cosserat}
		a((u,\omega),(v,\xi))=f(v,\xi),
	\end{flalign}
	where
	\begin{subequations}
		\label{eq:def_primal_cosserat}
		\begin{flalign}
        \begin{split}
      a((u,\omega),(v,\xi))&:= \int_{\Omega} \left(\langle\grad u-\mskw \omega,\grad v-\mskw \xi\rangle_{C_1}+ \langle \grad\omega,\grad\xi\rangle_{C_2}\right)\,dx,
        \end{split}
			\\
			f(v,\xi) &:= \int_{\Omega} \left(\langle f_u,v\rangle+\langle f_\omega,\xi\rangle\right)\,dx+\int_{\Gamma_N}\left(\langle g_u,v\rangle+\langle g_\omega,\xi\rangle\right)\,ds.
		\end{flalign}
	\end{subequations}
\end{problem}
\begin{lemma}[Well-posedness]
	Let $|\Gamma_D|>0$ have positive measure, $u_D,\omega_D\in H^{\half}(\Gamma_D)$, $f_u,f_\omega\in [H^{-1}(\Omega)]^3$, and $g_u,g_\omega\in [H^{-1/2}(\Gamma_N)]^3$. Then Problem~\ref{prob:primal_cosserat} has a unique solution $(u,\omega)\in [H^1(\Omega)]^3\times [H^1(\Omega)]^3$ and there holds the stability estimate
	\begin{align*}		
    \|u\|_{H^1(\Omega)}\!+\!\|\omega\|_{H^1(\Omega)}\!\leq\! C\big(\|f_u\|_{H^{-1}(\Omega)}\!+\!\|f_\omega\|_{H^{-1}(\Omega)}\!+\!\|g_u\|_{H^{-\half}(\Gamma_N)}\!+\!\|g_\omega\|_{H^{-\half}(\Gamma_N)}\big),
	\end{align*}
	where the constant $C>0$ depends especially on the Lam\'e parameter $\lambda$, the Cosserat coupling constant $\mu_c$, and the aspect ratio of the domain $\Omega$.
\end{lemma}
\begin{proof}
	Follows by Lax-Milgram Lemma using Korn's inequality.
\end{proof}
The lack of robustness in the aspect ratio originates from Korn's inequality. For nearly incompressible materials, the Lam\'e parameter $\lambda$ goes to infinity. In this work, the Cosserat coupling constant is of special interest to us. Its limit $\mu_c\to\infty$ enforces the equality $\half\curl u=\omega$, cf. \eqref{eq:cosserat-energy}.

\begin{remark}[Couple stress limit problem]
	In the formal limit $\mu_c\to\infty$ we can eliminate $\omega$ by $u$ via $\half\curl u=\omega$. This yields the following fourth-order \emph{couple stress problem} \cite{PG2008,CLS2021}
    \begin{flalign*}
        -\div(\mathcal{C}(\sym\grad u)) -\half\curl\big(\div C_2(\half\grad\curl u)-f_{\omega}\big) = f_u.
	\end{flalign*}
	The investigation of stable mixed finite elements for the couple stress problem circumventing the necessity of $H^2$-conforming elements will be topic of future research.
\end{remark}

\subsubsection{Lagrange finite elements}
\label{subsec:primal_formulation}

Let $\mathcal{T}$ be a triangulation of the domain $\Omega\subset\R^3$ consisting of (possibly polynomially curved) tetrahedra. Denote for $T\in\mathcal{T}$ the set of polynomials up to order $k$ by $\mathcal{P}^k(T)$. Then $\mathcal{P}^k(\mathcal{T})=\Pi_{T\in\mathcal{T}}\mathcal{P}^k(T)$ is the set of all elementwise polynomials of order $k$ without any continuity assumption over element interfaces. We define the scalar-valued \emph{Lagrangian} finite element space as an $H^1$-conforming subspace by
\begin{flalign*}
    	\Lag^k&:=\{u_h\in\mathcal{P}^k(\mathcal{T})\,:\, u_h \text{ is continuous}\} = \mathcal{P}^k(\mathcal{T})\cap C^0(\Omega)\subset H^1(\Omega),\\
 \Lag^k_{\Gamma}&:=\{u_h\in \Lag^k\,:\, u=0 \mbox{ on } \Gamma\}\subset H^1_{\Gamma}(\Omega)
\end{flalign*}
and its vector-valued version $\vLag^k:=[\Lag^k]^3$.

A straightforward approach to discretize  \eqref{eq:variation_cosserat} is to use Lagrangian finite elements for both the displacement $u$ and the rotation $\omega$. 
In the following, we always assume that the prescribed Dirichlet data $u_D$ and $\omega_D$ are traces of the used finite element spaces. The primal discrete formulation reads:

\begin{problem}[Primal discretization method]
\label{prob:primal_lagrange_formulation}
Let $k\geq 1$ be an integer. Find $(u_h,\omega_h)\in \vLag^k\times\vLag^k$ such that $u_h=u_D$ and $\omega_h=\omega_D$ on $\Gamma_D$ and
 for all $(v_h,\xi_h)\in \vLag^k_{\Gamma_D}\times\vLag^k_{\Gamma_D}$
\begin{flalign}
\label{eq:primal_method}
a((u_h,\omega_h),(v_h,\xi_h))=f(v_h,\xi_h),
\end{flalign}
where $a(\cdot,\cdot)$ and $f(\cdot,\cdot)$ are defined as in \eqref{eq:def_primal_cosserat}.
\end{problem}
Well-posedness of \eqref{eq:primal_method} follows directly from the continuous level due to the conforming discretization approach $\Lag^k\subset H^1(\Omega)$. Using Cea's Lemma, we directly obtain the quasi-best approximation of the finite element solution and a-priori convergence estimates regarding the mesh-size $h$. The constant $C=C(\mu_c,\lambda,\Omega)$, however, depends on the Cosserat coupling constant $\mu_c$, Lam\'e parameter $\lambda$, and due to Korn's inequality on the aspect ratio of the domain $\Omega$
\begin{flalign*}
    \| (u_h,\omega_h) - (u,\omega)\|_{H^1\times H^1}&\le C \inf\limits_{(v_h,\xi_h)\in \vLag^k\times\vLag^k} \| (v_h,\xi_h) - (u,\omega)\|_{H^1\times H^1}\\
    &\leq C\,h^k\left(|u|_{H^{k+1}}+|\omega|_{H^{k+1}}\right) .
\end{flalign*}
We demonstrate in Section~\ref{sec:numerics} that the primal method suffers from locking for $\mu_c\to\infty$.

\section{Mixed TDNNS and MCS formulations for Cosserat elasticity}
\label{sec:function_spaces}

The key idea to achieve robust discretization with respect to  $\mu_{c}$ is that {\it the space of the rotations $\omega$ should be large enough to contain $\mskw^{\dagger}\circ \grad=\vskw \circ \grad=curl$ of the space of the displacement $u$}, where $\mskw^{\dagger}$ is the pseudo-inverse of $\mskw$
\clearpage
  \begin{equation}\label{idea-marked}
\begin{tikzcd}
u\arrow{r}{\grad} &~ \\
\omega \arrow{r}{\grad} \arrow[ur, "-\mskw"] & \couplestress'.
\end{tikzcd}
\end{equation}
\vbox{\vspace{-6.cm}
\leftline{\hspace{7.3cm}\vspace{0.3cm}{
\begin{tikzpicture}[scale=1.15, color=red]
\draw[very thick, opacity=.75, ->, rounded corners] (5.5,1.5) -- ++(1.4, 0.0) -- ++(-1.5,-1.2)  ;
\end{tikzpicture}}
}} 
We start with the natural choice $u\in [H^{1}(\Omega)]^{3}$ (correspondingly, the Lagrange space on the discrete level). Then $\curl u\in \curl ([H^{1}(\Omega)]^{3})\subset H(\div)$ (correspondingly, $\curl u$ is in the Raviart-Thomas finite element space on the discrete level). This motivates us to use the $H(\div)$ space for $\omega$. Then $\grad \omega$ is in a weak $H^{-1}$-based Sobolev space (a Dirac delta on the discrete level). We must introduce the \emph{couple stress} tensor $\couplestress = C_2(\grad\omega)$ in a proper space to accommodate the pair $\langle \grad \omega, m \rangle$. We used the notation $\couplestress'$ in \eqref{idea-marked} to indicate that 
$\couplestress'$ is in the dual space. In this case, the choice of continuous and discrete spaces for $m$ coincides with the \emph{mass conserving mixed stress (MCS)} method for Stokes equations \cite{gopalakrishnan2020mass0}. Therefore, we refer to this formulation as the {\it MCS method}. Note that $\couplestress$ is not necessarily symmetric. 
  
Although the above idea fixes the dependence on $\mu_{c}$, since we choose $u\in [H^{1}(\Omega)]^{3}$ as the displacement formulation like in classical elasticity, the numerical scheme still suffers from standard volume and shear locking in elasticity. This inspires us to run the above idea starting with $u\in H(\curl)$, mimicking the \emph{tangential-displacement normal-normal-stress (TDNNS)} method for elasticity, which was also used for linear elasticity \cite{pechstein2011tangential} and Reissner--Mindlin plates \cite{PS17}. Now $\grad u$ is also a distribution, and we define the \emph{elasticity stress} tensor $\sigma = \mathcal{C}(\sym\grad u)$ to evaluate $\grad u$. The stresses are settled in specific matrix-valued function spaces and finite element spaces introduced in the following sections. Introducing $\sigma$ and $m$ leads to a mixed scheme. We remark that here the idea leading to mixed formulations is different from \cite{boon2024mixedfiniteelementmethods}, which explicitly used the last two spaces in the twisted de~Rham complex, mimicking the formulation for the mixed Poisson problem. 

\subsection{Continuous setting}
For matrix fields, we use the convention that differential operators act row-wise. For simplicity of presentation, we focus in this section on the boundary conditions corresponding to \eqref{elasticity-diagram-Hd0}, i.e., we impose $u=\omega=0$ on the boundary. Following \cite{gopalakrishnan2020mass0,pechstein2011tangential}, we define the following matrix-valued function spaces
\begin{flalign}
H(\curl\div, \M)&:=\{\couplestress \in [L^{2}(\Omega)]^{3\times 3}: \curl\div \couplestress\in [H^{-1}(\Omega)]^{3}\}\nonumber\\ 
&=\{\couplestress \in [L^{2}(\Omega)]^{3\times 3}: \div \couplestress\in H_{0}(\div)^{\ast}\},\label{eq:hcurldiv}\\
H(\div\div, \S)&:=\{\sigma \in [L^{2}(\Omega)]^{3\times 3}_{\sym}: \div \div \sigma\in H^{-1}(\Omega)\}\nonumber\\
&=\{\sigma \in [L^{2}(\Omega)]^{3\times 3}_{\sym}: \div \sigma\in H_{0}(\curl)^{\ast}\},\label{eq:hdivdiv}
\end{flalign}
where we used that the dual space $H_0(\div)^{\ast}=H^{-1}(\curl):=\{u\in [H^{-1}(\Omega)]^3\,:\, \curl u\in [H^{-1}(\Omega)]^3\}$ and analogously that $H_0(\curl)^{\ast}=H^{-1}(\div)$. Note that the notation is different from some 
literature where the spaces $H(\curl\div, \M)$ and $H(\div\div, \S)$ contain $L^{2}$ fields with second derivatives in $L^{2}$ instead of $H^{-1}$.

The first space \eqref{eq:hcurldiv} represents the stress space on the continuous level of the MCS method for Stokes equations. Due to its definition, the duality-pairing
\begin{flalign}
	\label{eq:duality_curldiv}
	\langle \div\couplestress,\omega\rangle_{H(\div)^{\ast}}\!:= \!\langle \div\couplestress,\omega\rangle_{H_0(\div)^{\ast}\times H_{0}(\div)}\!=\!-\langle \grad \omega,\couplestress\rangle_{H(\curl\div)^{\ast}\times H(\curl\div)}
\end{flalign}
is well-defined for all $\couplestress\in H(\curl\div, \M)$ and $\omega\in H_{0}(\div)$. Space \eqref{eq:hdivdiv} is used for the moment stress tensor in the Hellan--Herrmann--Johnson (HHJ) \cite{Hel67,Her67,Joh73} method and the elasticity stress tensor in the TDNNS method for linear elasticity and Reissner--Mindlin plates. 
Analogously, the following duality-pairing is well-defined for all $\sigma\in H(\div\div, \S)$ and $u\in H_{0}(\curl)$
\begin{flalign}
	\label{eq:duality_divdiv}
	\langle \div\sigma,u\rangle_{H(\curl)^{\ast}}:= \langle \div\sigma,u\rangle_{H_0(\curl)^{\ast}\times H_{0}(\curl)}=-\langle \grad u,\sigma\rangle_{H(\div\div)^{\ast}\times H(\div\div)}.
\end{flalign}

We propose the following two mixed formulations for Cosserat elasticity.
\begin{problem}[MCS and TDNNS mixed formulations for linear Cosserat elasticity]
\label{prob:dual_mixed_mcs_tdnns}
Find $(u,\omega,\couplestress)\in [H_{0}^1(\Omega)]^3\times H_{0}(\div)\times H(\curl\div,\M)$ solving the Lagrangian 
\begin{flalign}
	\label{eq:mixed_MCS}
 \begin{split}
 		\mathcal{L}^{\couplestress}(u,\omega,\couplestress) &= \frac{1}{2}\int_{\Omega}\big(\|\grad u-\mskw \omega\|_{C_1}^{2}-\|\couplestress\|_{C_2^{-1}}^{2}\big)\,dx\\
        &\quad-\langle \div\couplestress,\omega\rangle_{H(\div)^{\ast}}- f(u,\omega,\couplestress)\to \min_{u,\omega}\max_{\couplestress}.
   \end{split}
 	\end{flalign}
Find $(u,\omega,\couplestress,\sigma)\in H_{0}(\curl)\times H_{0}(\div)\times H(\curl\div,\M)\times H(\div\div,\S)$ solving the Lagrangian
\begin{flalign}
	\label{eq:mixed_TDNNS_MCS}
		\begin{split}
			\mathcal{L}^{\couplestress,\sigma}(u&,\omega,\sigma,\couplestress) = \frac{1}{2}\int_{\Omega}\left(-\|\sigma\|^2_{\mathcal{C}^{-1}}+2\mu_c\left\|\half\curl u- \omega\right\|^{2}-\|\couplestress\|_{C_2^{-1}}^{2}\right)\,dx\\
			&\quad-\langle \div\sigma,u\rangle_{H(\curl)^{\ast}}-\langle \div\couplestress,\omega\rangle_{H(\div)^{\ast}}- f(u,\omega,\couplestress,\sigma)\to\min_{u,\omega}\max_{\couplestress,\sigma}.
		\end{split}
\end{flalign}
Here, the inverse of $\mathcal{C}^{-1}$ is the usual \emph{compliance} tensor of linear elasticity.
\end{problem}
The MCS formulation \eqref{eq:mixed_MCS} is based on the following diagram:
 \begin{equation*}
\begin{tikzcd}
\left[H_{0}^{1}(\Omega)\right]^{3}\arrow{r}{\grad} &~ \\
H_{0}(\div) \arrow{r}{\grad} \arrow[ur, "-\mskw"] & H(\curl\div, \M)^{\ast}
\end{tikzcd}, \qquad\qquad 
\begin{tikzcd}
u\arrow{r}{\grad} &~ \\
\omega \arrow{r}{\grad} \arrow[ur, "-\mskw"] & \couplestress',
\end{tikzcd}
 \end{equation*}
 where we used the notation $\couplestress'$ to indicate that $\couplestress$ is introduced in the dual space. The TDNNS-MCS formulation \eqref{eq:mixed_TDNNS_MCS} is based on the diagram 
 \begin{equation*}
\begin{tikzcd}
H_{0}(\curl)\arrow{r}{\grad} &H(\div\div, \S)^{\ast}\times [L^2(\Omega)]^{3\times 3}_{\skw} \\
H_{0}(\div) \arrow{r}{\grad} \arrow[ur, "-\mskw"] & H(\curl\div, \M)^{\ast}
\end{tikzcd}, \qquad\qquad 
\begin{tikzcd}
u\arrow{r}{\grad} &\sigma' \\
\omega \arrow{r}{\grad} \arrow[ur, "-\mskw"] & \couplestress'.
\end{tikzcd}
 \end{equation*}
Similarly, we use  $\sigma' $ and $\couplestress'$ to indicate that $\sigma$ and $\couplestress$ are in the dual spaces.

We recall that in the formal limit $\mu_c\to\infty$ we establish the constraint
\begin{flalign}
	\label{eq:constraint}
	\half\curl u=\omega.
\end{flalign}
Thus, a necessary condition to avoid locking 
in the finite element computations is that the chosen function spaces can represent \eqref{eq:constraint} exactly. In the primal setting \eqref{eq:variation_cosserat}, where $(u,\omega)\in [H^1(\Omega)]^3\times [H^1(\Omega)]^3$, the constraint is obviously not fulfilled. 
This indicates that locking might occur after discretization. The curl of a function in $[H^1(\Omega)]^3$ or $H(\curl)$ is in $H(\div)$. 
Hence, seeking the rotation field $\omega$ in $H(\div)$ is a natural choice and fulfills the constraint \eqref{eq:constraint} exactly. 
This indicates that the mixed formulations \eqref{eq:mixed_MCS} and \eqref{eq:mixed_TDNNS_MCS} have the potential to be robust with respect to the Cosserat coupling constant $\mu_c$, if the discretization is done carefully. The use of the TDNNS formulation in \eqref{eq:mixed_TDNNS_MCS} yields further beneficial properties:  
The proposed method becomes 
more robust with respect to anisotropic domains with large aspect ratio, i.e., avoids shear locking \cite{PS12}. Furthermore, 
by adding a stabilization term in the discrete setting, our method will be free of volumetric locking when $\lambda\to\infty$, i.e. in the incompressible limit  \cite{pechstein2011tangential}.
\subsection{Finite element spaces}
\label{subsec:fe_spaces}

We define several vector- and matrix-valued finite element spaces that will be used below. The \emph{Raviart--Thomas} (RT) \cite{Raviart.P;Thomas.J.1977a} and the \emph{Brezzi--Douglas--Marini} (BDM)  elements \cite{brezzi1985two} are used  
to discretize $H(\div)$ and the \emph{N\'ed\'elec} elements (of first and second kind) \cite{Nedelec.J.1986a} 
are used to discretize $H(\curl)$. They fit in the de~Rham complexes \cite{arnold2018finite,Arnold.D;Falk.R;Winther.R.2006a,Arnold.D;Falk.R;Winther.R.2010a}. For simplicity of presentation, we only use the RT and the N\'ed\'elec element of the second kind below, defined by 
\begin{align*}
&\RT^{k}:=\{u\in H(\div)\,:\, \forall T\in \mathcal{T}\,u|_{T}=a+bx, a \in [\mathcal{P}^{k}(T)]^{3}, b\in \tilde{\mathcal{P}}^{k}(T)\}\subset H(\div), \\
&\RT^{k}_{\Gamma}:= \{u_h\in \RT^k\,:\, \langle u_h, n\rangle = 0 \mbox{ on } \Gamma\} \subset H_{\Gamma}(\div),\\
&\NedII^{k}:=\{u\in H(\curl)\,:\, u|_{T}\in [\mathcal{P}^{k}(T)]^{3}, ~\forall T\in \mathcal{T}\} \subset H(\curl),\\
&{\NedII^{k}}_{,\Gamma}:= \{u_h\in \NedII^k\,:\, u_h\times n = 0 \mbox{ on } \Gamma\} \subset H_{\Gamma}(\curl),
\end{align*}
where the set of elementwise homogeneous polynomials of order $k$ is denoted by $\tilde{\mathcal{P}}^k(\mathcal{T})$.  
In the lowest order cases, $\RT^{0}$ consists of facet-based shape functions and $\NedII^{1}$ consists of edge-based basis functions. We have the following de~Rham complex for $k\geq 3$
$$
\begin{tikzcd}
0 \arrow{r} & 	\Lag^k \arrow{r}{\grad} &\NedII^{k-1} \arrow{r}{\curl} &\RT^{k-2} \arrow{r}{\div} &\mathcal{P}^{k-3}(\mathcal{T}) \arrow{r} &0.
\end{tikzcd}
$$

Let $\mathcal{F}$ denote the set of facets of $\mathcal{T}$. The MCS element \cite{gopalakrishnan2020mass0} has tangential-normal continuity and is a slightly non-conforming subspace of $H(\curl\div,\M)$ 
\begin{flalign}
\label{eq:mcs_fes}
\MCS^{k}:=\{\sigma_h\in  [\mathcal{P}^k(\mathcal{T}) ]^{3\times 3}\,:\, 
 n_{F}\times (\sigma_h n_{F})\text{ is continuous across $F\in\mathcal{F}$}\}. 
\end{flalign}

The TDNNS \cite{pechstein2011tangential} and HHJ methods discretize the stress tensor using symmetric piecewise polynomials with normal-normal continuity. We denote this space by
\begin{equation*}
\HHJ^{k}:=\{\sigma_h\in [\mathcal{P}^k(\mathcal{T}) ]^{3\times 3}_{\sym}\,:\, \sigma_{h,n_Fn_F} :=\langle \sigma_hn_{F}, n_{F}\rangle\text{ is continuous across $F\in\mathcal{F}$}\},
\end{equation*}
which is a slightly non-conforming subspace of $H(\div\div,\S)$.
The definitions of $\MCS^k_{\Gamma}$ and $\HHJ^k_{\Gamma}$ are done analogously to the definitions of $\RT^k_{\Gamma}$ and ${\NedII^k}_{,\Gamma}$ by forcing the tangential-normal and normal-normal trace to be zero on $\Gamma$, respectively.

Local shape functions and interelement continuity characterize the above spaces. Possible degrees of freedom can be found in the references above. The following observation shows that the constraint \eqref{eq:constraint} is also fulfilled in the discrete setting.
\begin{lemma}
	\label{lem:cont_curl}
 Let $u\in \{\vLag_{\Gamma}^k,\, {\NedII}_{,\Gamma}^k\}$. Then $\curl u\in \RT_{\Gamma}^{k-1}$.
\end{lemma}
\begin{proof}
	The claim
    follows by noting that $\curl u$ is normal continuous for Lagrangian and N\'ed\'elec elements and fulfills the zero boundary conditions.
\end{proof}

\subsection{Mass conserving mixed stress (MCS) formulation for rotational part}
\label{subsec:mcs_formulation}

For $u_h\in \vLag^k$ its curl is in $\RT^{k-1}$. 
Therefore, setting $\omega_h\in\RT^{k-1}$ enables us to satisfy the constraint $\omega_h=\half\curl u_h$ exactly, such that, as we will prove in Section~\ref{sec:stability_error}, 
our method will be locking-free for $\mu_c\to\infty$. 
In \cite{gopalakrishnan2020mass0} it was shown that for $\omega_h\in\RT^{k-1}$ and $\couplestress_h\in\MCS^{k-1}$  \eqref{eq:duality_curldiv} has the explicit mesh dependent form
\begin{flalign}
	\langle \grad \omega_h,\couplestress_h&\rangle_{H(\curl\div)^{\ast}}=\sum_{T\in\mathcal{T}}\Big(\int_T\langle\grad \omega_h,\couplestress_h\rangle\,dx-\int_{\partial T}\langle\omega_{h,t}, \couplestress_{h,nt}\rangle\,ds\Big)\nonumber\\
	&=-\sum_{T\in\mathcal{T}}\Big(\int_T \langle\omega_h,\div\couplestress_h\rangle\,dx-\int_{\partial T}\omega_{h,n} \couplestress_{h,nn}\,ds\Big) = -\langle \div\couplestress_h,\omega_h\rangle_{H(\div)^{\ast}}.\label{eq:duality_RT_MCS}
\end{flalign}
Above, $n$ is the outer unit normal vector of element $T$ and $\omega_{h,t}:= (I-n\otimes n)\omega_h$ denotes the projection onto the tangent space. 
We pose the \emph{MCS discretization for Cosserat elasticity} as follows.
\begin{problem}
	\label{prob:mcs_method}
Let $k\geq 1 $ be an integer. Find $(u_h,\omega_h,\couplestress_h)\in \vLag^k\times\RT^{k-1}\times\MCS^{k-1}$ such that $u_h=u_D$ and $\omega_{h,n}=\omega_{D,n}$ on $\Gamma_D$, $\couplestress_{h,t}=g_{\omega,t}$ on $\Gamma_N$, and
 for all $(v_h,\xi_h,\Psi_h)\in \vLag^k_{\Gamma_D}\times\RT^{k-1}_{\Gamma_D}\times\MCS^{k-1}_{\Gamma_N}$,
\begin{subequations}
	\label{eq:mcs_method}
	\begin{alignat}{3}
		&a(\couplestress_h,\Psi_h) &+& b(\Psi_h,(u_h,\omega_h))&=&g(\Psi_h),\\
		&b(\couplestress_h,(v_h,\xi_h))&-&c((u_h,\omega_h),(v_h,\xi_h)) &=&-f(v_h,\xi_h),
	\end{alignat}
\end{subequations}
where
\begin{subequations}
	\begin{flalign}
		\label{eq:def_mcs}
			&a(\couplestress,\Psi):=\int_{\Omega} \langle\couplestress,\Psi\rangle_{C_2^{-1}}\,dx,\qquad b(\Psi,(u,\omega)):=\langle\omega,\div\Psi\rangle_{H(\div)^{\ast}},\\
			&c((u,\omega),(v,\xi)) := \int_{\Omega}2\mu_c\left\langle\half\curl u-\omega,\half\curl v-\xi\right\rangle\,dx,\label{eq:def_mcs_c}\\
			&f(v,\xi):=\int_{\Omega} (\langle f_u, v\rangle+ \langle f_\omega, \xi\rangle)\,dx+ \int_{\Gamma_N}(\langle g_u, v\rangle+g_{\omega,n}\xi_{n})\,ds,\\
      & g(\Psi):=\int_{\Gamma_D}\langle\Psi_{nt},\omega_{D,t}\rangle\,ds.
	\end{flalign}
\end{subequations}
\end{problem}
The numerical analysis of Problem~\ref{prob:mcs_method} will follow from the (more involved) analysis of the method presented in the next section, cf. Corollary~\ref{cor:stability_convergence_mcs}. The Dirichlet and Neumann boundary conditions $\omega_D$, $g_{\omega}$ are treated  
as half essential and half natural boundary conditions. This comes from the shared regularity between the rotation $\omega_h$ and couple stress tensor $\couplestress_h$.  
The change of essential and natural boundary conditions between primal and mixed formulations is classical in mixed methods, see, e.g. \cite{boffi2013mixed}.

Problem~\ref{prob:mcs_method} has the form of a saddle-point problem, leading to an indefinite stiffness matrix after assembling. We can apply hybridization techniques to overcome this problem \cite{boffi2013mixed}. For the sake of simplicity  
we refer to the literature for details \cite{gopalakrishnan2020mass}.

\subsection{Tangential-displacement normal-normal-stress (TDNNS) formulation for displacement part}
\label{subsec:tdnns_formulation}
Due to Korn's inequality, the standard elasticity formulation is not robust with respect to anisotropic structures. Further, it is well-known that locking can occur in the nearly incompressible regime when $\lambda\to\infty$. The TDNNS method \cite{pechstein2011tangential} is robust in both cases. We will use it in addition to the previously introduced MCS method to discretize the elasticity part. To this end, the displacement fields are settled in the N\'ed\'elec space, $u_h\in\NedII^k$. There still holds  $\curl u_h\in \RT^{k-1}$, cf. Lemma~\ref{lem:cont_curl}, such that constraint \eqref{eq:constraint} holds when using $\omega_h\in\RT^{k-1}$. For $u_h\in\NedII^k$ and $\sigma_h\in \HHJ^k$  \eqref{eq:duality_divdiv} is well-defined and takes the explicit form \cite{pechstein2011tangential}, cf. \eqref{eq:duality_RT_MCS},
\begin{flalign*}
		\langle \div\sigma_h,u_h\rangle_{H(\curl)^{\ast}}=\sum_{T\in\mathcal{T}}\Big(\int_T \langle u_h,\div\sigma_h\rangle\,dx+\int_{\partial T}\langle u_{h,t}, \sigma_{h,nt}\rangle\,ds\Big).
\end{flalign*}
We define the \emph{TDNNS-MCS discretization method for Cosserat elasticity} as follows. 
\begin{problem}
	\label{prob:mcs_tdnns_method}
	Let $k\geq 1$ be an integer. Find $(u_h,\omega_h,\couplestress_h,\sigma_h)\in \NedII^k\times\RT^{k-1}\times\MCS^{k-1}\times\HHJ^k$ such that $u_{h,t}=u_{D,t}$, $\omega_{h,n}=\omega_{D,n}$ on $\Gamma_D$, $\couplestress_{h,nt}=g_{\omega,t}$, $\sigma_{h,nn}=g_{u,n}$ on $\Gamma_N$, and 
	for all $(v_h,\xi_h,\Psi_h,\Theta_h)\in {\NedII^k}_{,\Gamma_D} \times\RT^{k-1}_{\Gamma_D} \times\MCS^{k-1}_{\Gamma_N} \times\HHJ^k_{\Gamma_N}$,
\begin{subequations}
	\label{eq:mcs_tdnns_method}
	\begin{alignat}{3}
		&a((\couplestress_h,\sigma_h),(\Psi_h,\Theta_h)) &+& b((\Psi_h,\Theta_h),(u_h,\omega_h)) &=&g(\Psi_h,\Theta_h),\\
		&b((\couplestress_h,\sigma_h),(v_h,\xi_h))&-&c((u_h,\omega_h),(v_h,\xi_h))&=&-f(v_h,\xi_h),
	\end{alignat}
\end{subequations}
where $c(\cdot,\cdot)$ is defined as in \eqref{eq:def_mcs_c} and
	\begin{flalign*}
		a((\couplestress,\sigma),(\Psi,\Theta))&:=\int_{\Omega}\left( \langle\couplestress,\Psi\rangle_{C_2^{-1}}+\langle\sigma,\Theta\rangle_{\mathcal{C}^{-1}}\right)\,dx,\\
		b((\Psi,\Theta),(u,\omega))&:=\langle\omega,\div\Psi\rangle_{H(\div)^{\ast}}+\langle u,\div\Theta\rangle_{H(\curl)^{\ast}},\\
		f(v,\xi)&:=\int_{\Omega}( \langle f_u, v\rangle+ \langle f_\omega, \xi\rangle)\,dx+ \int_{\Gamma_N}(\langle g_{u,t}, v_{t}\rangle+g_{\omega,n}\xi_{n})\,ds,\\
  g(\Psi,\Theta)&:=\int_{\Gamma_D}\left(\langle\Psi_{nt},\omega_{D,t}\rangle+\Theta_{nn}u_{D,n}\right)\,ds.
	\end{flalign*}
\end{problem}
In \eqref{eq:mcs_tdnns_method}, additionally to $\omega_D$ and $g_{\omega}$, also $u_D$ and $g_{u}$ are treated half-half as essential and natural boundary conditions. The system in $(u_h,\sigma_h)$ has the form of a saddle-point problem, very similar to the TDNNS method. We can again use hybridization techniques to retain a symmetric and positive definite problem,
cf. \cite{boffi2013mixed, NPS2021,pechstein2011tangential,Sin2009}.

\section{Stability, error analysis, and post-processing}
\label{sec:stability_error}

For the numerical analysis, we assume throughout this section that homogeneous Dirichlet boundary conditions are prescribed on the whole boundary. We emphasize that the extension to Neumann and non-homogeneous boundary conditions is straightforward.

\subsection{Stability}
To show that \eqref{eq:mcs_tdnns_method} is well-posed and robust with respect to the Cosserat coupling constant $\mu_c$ we follow the strategy of \cite{PS17}. 
First, we add the unknown $\gamma:=2\mu_c(\half\curl u-\omega)\in H(\div)$, $\gamma_h=2\mu_c(\half\curl u_h-\omega_h)\in\RT^{k-1}$, which can be interpreted as a \emph{shear stress} quantity. Problem~\ref{prob:mcs_tdnns_method} is then equivalent to: 
\begin{problem}
	\label{prob:mcs_tdnns_method_gamma}
	Let $k\geq 1$ be an integer. Find $(u_h,\omega_h,\couplestress_h,\sigma_h,\gamma_h)\in {\NedII}_{,0}^k\times{\RT}_{0}^{k-1}\times\MCS^{k-1}\times\HHJ^k\times{\RT}_{0}^{k-1}$ such that for all $(v_h,\xi_h,\Psi_h,\Theta_h,\delta_h)\in {\NedII}_{,0}^k\times{\RT}_{0}^{k-1}\times{\MCS}^{k-1}\times{\HHJ}^k\times{\RT}_{0}^{k-1}$
	\begin{subequations}
		\label{eq:mcs_tdnns_method_gamma}
		\begin{alignat}{3}
		 	&a((\couplestress_h,\sigma_h,\gamma_h),(\Psi_h,\Theta_h,\delta_h)) &+& b((\Psi_h,\Theta_h,\delta_h),(u_h,\omega_h))  &=&0,\\
		 	&b((\couplestress_h,\sigma_h,\gamma_h),(v_h,\xi_h))&& &=&-f(v_h,\xi_h),
		 \end{alignat}
	\end{subequations}
	where $f$ is defined as in Problem~\ref{prob:mcs_tdnns_method} and
	\begin{subequations}
		\label{eq:def_mcs_tdnns_gamma}
		\begin{flalign}
			&a((\couplestress,\sigma,\gamma),(\Psi,\Theta,\delta)):=\int_{\Omega}\left( \langle\couplestress,\Psi\rangle_{C_2^{-1}}+ \langle\sigma,\Theta\rangle_{\mathcal{C}^{-1}}+ \frac{1}{2\mu_c}\langle\gamma,\delta\rangle\right)\,dx,\\
			b((\couplestress&,\sigma,\gamma),(u,\omega))\!:=\!\langle \div \sigma,u\rangle_{H(\curl)^{\ast}}\!+\!\langle \div\couplestress,\omega\rangle_{H(\div)^{\ast}}\!-\!\int_{\Omega}\left\langle \half\curl u\!-\!\omega,\gamma\right\rangle\,dx.
		\end{flalign}
	\end{subequations}
\end{problem}
Equivalence follows  
immediately, 
as we can eliminate $\gamma_h$ algebraically  
by our choice of finite element spaces. So the smaller system should be used for implementation.  
First, we prove well-posedness of the proposed discrete mixed problems. 
\begin{theorem}
	\label{thm:robustness_tdnns}
	Problem~\ref{prob:mcs_tdnns_method_gamma}, and therefore Problem~\ref{prob:mcs_tdnns_method}, is well-posed and with $\gamma_h=2\mu_c(\half\curl u_h-\omega_h)$ 
    the following stability estimate holds
	\begin{flalign}
		\begin{split}
      \label{eq:robust_estimate}
      		\|\couplestress_h\|_{L^2} + \|\sigma_h\|_{L^2} + \|\gamma_h\|_{\Gamma}+ \|u_h\|_{V_h} + \|\omega_h\|_{W_h}+\sqrt{\mu_c}&\left\|\half\curl u_h-\omega_h\right\|_{L^2} \\
          &\leq C (\|f_u\|_{L^2}+\|f_{\omega}\|_{L^2}),
    \end{split}
	\end{flalign}
	where $C>0$ is a constant independent of $\mu_c$ and the norms $\|\cdot\|_{\Gamma}$, $\|\cdot\|_{V_h}$, and $\|\cdot\|_{W_h}$ are given by, $[\![\cdot]\!]$ denoting the jump over facets,
	\begin{flalign*}
		\|u\|_{V_h}^2&= \sum_{T\in\mathcal{T}}\|\sym\grad u\|^2_{L^2(T)} + \frac{1}{h}\sum_{F\in\mathcal{F}}\|[\![u_{n}]\!]\|_{L^2(F)}^2,\qquad\qquad \|\gamma\|_{\Gamma}= \frac{1}{\sqrt{\mu_c}}\|\gamma\|_{L^2},\\
  \|\omega\|_{W_h}^2&= \sum_{T\in\mathcal{T}}\|\grad \omega\|^2_{L^2(T)} + \frac{1}{h}\sum_{F\in\mathcal{F}}\|[\![\omega_{t}]\!]\|_{L^2(F)}^2.
	\end{flalign*}
\end{theorem}
\begin{proof}
	We use Brezzi's theorem of saddle point problems \cite{brezzi1974existence}. 
 On the spaces
	\begin{flalign*}
		X_h:= \HHJ^k\times \MCS^{k-1}\times \RT^{k-1},\qquad Y_h:= \NedII^k\times \RT^{k-1}
	\end{flalign*}
	we define the following parameter depending norms
	\begin{flalign*}
		&\|(\couplestress,\sigma,\gamma)\|^2_{X_h} := \|\couplestress\|_{L^2}^2 + \|\sigma\|_{L^2}^2 + \|\gamma\|_{\Gamma}^2,\\
		& \|(u,\omega)\|^2_{Y_h} : = \|u\|_{V_h}^2 + \|\omega\|_{W_h}^2+\mu_c\|\half\curl u-\omega\|^2_{L^2}.
	\end{flalign*}
	We denote the lower and upper limit of the material tensors for $\sigma_h\in\HHJ^k$ and $\couplestress_h\in\MCS^{k-1}$ by
	\begin{flalign*}
		&\underline{c}_1\|\sigma_h\|^2\leq \|\sigma_h\|^2_{\mathcal{C}^{-1}}\leq \overline{c}_1\|\sigma_h\|^2,\qquad \underline{c}_2\|\couplestress_h\|^2\leq \|\couplestress_h\|^2_{C_2^{-1}}\leq \overline{c}_2\|\couplestress_h\|^2.
	\end{flalign*}
	The continuity of the right-hand side, $|f(v_h,\xi_h)|\leq (\|f_u\|_{L^2}+\|f_{\omega}\|_{L^2})\|(v_h,\xi_h)\|_{Y_h}$, is straightforward.
	Continuity of bilinear form $a(\cdot,\cdot):{X_h}\times {X_h}\to \R$ follows directly by the Cauchy-Schwarz inequality. For all $(\couplestress_h,\sigma_h,\gamma_h),(\Psi_h,\Theta_h,\delta_h)\in {X_h}$ there holds
	\begin{flalign*}
		|a((\couplestress_h,\sigma_h,\gamma_h),(\Psi_h,\Theta_h,\delta_h))|
		&\leq 3\max\{\overline{c}_1,\overline{c}_2,\half\}\|(\couplestress_h,\sigma_h,\gamma_h)\|_{X_h}\|(\Psi_h,\Theta_h,\delta_h)\|_{X_h}.
	\end{flalign*}
For the continuity of bilinear form $b(\cdot,\cdot):{X_h}\times {Y_h}\to \R$ we use that 
	\begin{flalign*}
		\langle\div\sigma_h,u_h\rangle_{H(\curl)^{\ast}}\leq \overline{c}_{\mathrm{T}}\|\sigma_h\|_{L^2}\|u_h\|_{V_h}\,  \text{ and } \,\langle\div\couplestress_h,\omega_h\rangle_{H(\div)^{\ast}}\leq \overline{c}_{\mathrm{M}}\|\couplestress_h\|_{L^2}\|\omega_h\|_{W_h}
	\end{flalign*} 
	have been proven in the analysis of the TDNNS \cite{pechstein2011tangential} and MCS \cite{gopalakrishnan2020mass0} methods. There holds for all $(\couplestress_h,\sigma_h,\gamma_h)\in {X_h}$ and $(u_h,\omega_h)\in {Y_h}$
	\begin{flalign*}
		|b((\couplestress_h,\sigma_h,\gamma_h),(u_h,\omega_h))|
		&\leq 3\max\{\overline{c}_{\mathrm{T}},\overline{c}_{\mathrm{M}},1\}\|(\couplestress_h,\sigma_h,\gamma_h)\|_{X_h}\|(u_h,\omega_h)\|_{Y_h}.
	\end{flalign*}
 Bilinear form $a(\cdot,\cdot)$ is coercive, i.e., there holds for all $(\couplestress_h,\sigma_h,\gamma_h)\in {X_h}$
	\begin{flalign*}
		 a((\couplestress_h,\sigma_h,\gamma_h),(\couplestress_h,\sigma_h,\gamma_h))
		 \geq \min\left\{ \underline{c}_1,\underline{c}_2,\half\right\}\|(\couplestress_h,\sigma_h,\gamma_h)\|_{X_h}^2.
	\end{flalign*}
    It remains to check the Ladyzhenskaya--Babu{\v{s}}ka--Brezzi (LBB) condition. Let an element $(u_h,\omega_h)\in {Y_h}$ be given. Choose $\tilde\gamma_h:=-\sqrt{\mu_c}(\half\curl u_h-\omega_h)$, and $\tilde\couplestress_h\in\MCS^{k-1}$ and $\tilde\sigma_h\in\HHJ^k$ such that, see \cite{pechstein2011tangential,gopalakrishnan2020mass0} for the TDNNS and MCS LBB conditions,
	\begin{flalign*}
		\langle \div\tilde\couplestress_h,\omega_h\rangle_{H(\div)^{\ast}}\geq \underline{c}_{\mathrm{M}} \|\tilde\couplestress_h\|_{L^2}\|\omega_h\|_{W_h},\quad \langle \div\tilde\sigma_h,u_h\rangle_{H(\curl)^{\ast}}\geq \underline{c}_{\mathrm{T}} \|\tilde\sigma_h\|_{L^2}\|u_h\|_{V_h}.
	\end{flalign*}
 We note that the choice of polynomial degree for the trace part of $\tilde \couplestress_h$ is crucial for the MCS LBB condition, which takes the role of the pressure here. Then, we have
	\begin{flalign*}
		&\sup\limits_{(\couplestress_h,\sigma_h,\gamma_h)\in {X_h}}\frac{b((\couplestress_h,\sigma_h,\gamma_h),(u_h,\omega_h))}{\|(\couplestress_h,\sigma_h,\gamma_h)\|_{X_h}} \geq \frac{1}{3}\left( \sup\limits_{\couplestress_h\in \MCS^{k-1}}\frac{b((\couplestress_h,0,0),(u_h,\omega_h))}{\|\couplestress_h\|_{L^2}}\right.\\
		&\qquad\qquad+\left.\sup\limits_{\sigma_h\in \HHJ^{k}}\frac{b((0,\sigma_h,0),(u_h,\omega_h))}{\|\sigma_h\|_{L^2}}+\sup\limits_{\gamma_h\in \RT^{k-1}}\frac{b((0,0,\gamma_h),(u_h,\omega_h))}{\|\gamma_h\|_{\Gamma}} \right)\\
		&\qquad\geq \frac{1}{3}\left( \frac{\langle \div\tilde\couplestress_h,\omega_h\rangle_{H(\div)^{\ast}}}{\|\tilde\couplestress_h\|_{L^2}}+\frac{\langle \div\tilde\sigma_h,u_h\rangle_{H(\curl)^{\ast}}}{\|\tilde\sigma_h\|_{L^2}}+\frac{\sqrt{\mu_c}\|\half\curl u_h-\omega_h\|_{L^2}^2}{\|\half\curl u_h-\omega_h\|_{L^2}}\right)\\
		&\qquad\geq \min\left\{ \nicefrac{1}{3},\underline{c}_{\mathrm{M}},\underline{c}_{\mathrm{T}} \right\}\|(u_h,\omega_h)\|_{Y_h}.
	\end{flalign*}
	By Brezzi's theorem, Problem~\ref{prob:mcs_tdnns_method_gamma} (and therefore Problem~\ref{prob:mcs_tdnns_method}) is well-posed and stability estimate \eqref{eq:robust_estimate} holds.  
    Since all involved constants $\underline{c}_1,\underline{c}_2,\overline{c}_1,\overline{c}_2,\underline{c}_{\mathrm{M}},\underline{c}_{\mathrm{T}}, \overline{c}_{\mathrm{M}}$, and $\overline{c}_{\mathrm{T}}$ are independent of $\mu_c$ the constant $C$ in \eqref{eq:robust_estimate} is also independent of $\mu_c$.
\end{proof}
Some comments are in place. By estimate \eqref{eq:robust_estimate} we have shown a robust a-priori estimate with respect to the Cosserat coupling constant $\mu_c$. Further, we obtain that the rotation field $\omega_h$ converges to the curl of the displacement $u_h$ as $\mu_c\to\infty$ 
\begin{flalign*}
	\left\|\half\curl u_h-\omega_h\right\|_{L^2}\leq \nicefrac{C}{\sqrt{\mu_c}},\qquad C\neq C(\mu_c).
\end{flalign*}

\begin{remark}[Robustness with respect to anisotropic structures and nearly incompressibility]
	\label{rem:robustness_anisotropy_incompressibility}
	Throughout the proof, including the cited inequalities, we never used Korn's inequality, where the resulting constant would depend on the aspect ratio of the domain. This indicates that the discrete methods will be free of shear locking. Nevertheless, a detailed convergence analysis would require hexahedral elements and anisotropic estimates, which are out of the scope of this paper. We refer to \cite{PS12} for an anisotropic analysis for the TDNNS method. In the presented proof, the Lam\'e constant $\lambda$ enters the coercivity constant of $a(\cdot,\cdot)$. With the stabilization techniques presented in \cite{Sin2009}, one can 
    adapt the proof to obtain a constant independent of $\lambda$, i.e. robustness in the nearly incompressible limit $\lambda\to\infty$.
\end{remark}

Before we can prove convergence, we need the following consistency result. 
\begin{theorem}[Consistency \& Galerkin orthogonality]
	\label{thm:consistency}
	Assume the exact solution of Cosserat problem \eqref{eq:variation_cosserat} fulfills the following regularity assumptions
	\[
		 u\in [H^1(\Omega)]^3,\quad \omega\in [H^1(\Omega)]^3, \quad \sigma \in [H^1(\Omega)]^{3\times 3},\quad \couplestress\in [H^1(\Omega)]^{3\times 3}.
	\] 
	Problems~\ref{prob:mcs_method}, \ref{prob:mcs_tdnns_method}, and \ref{prob:mcs_tdnns_method_gamma} are consistent. For example, for Problem~\ref{prob:mcs_tdnns_method_gamma}
    we have for all $(u_h,\omega_h,\couplestress_h,\sigma_h,\gamma_h)\in {\NedII}_{,0}^k\times\RT^{k-1}_{0}\times\MCS^{k-1}\times\HHJ^k\times\RT^{k-1}$
	\begin{flalign*}
		a((\couplestress,\sigma,\gamma),(\couplestress_h,\sigma_h,\gamma_h))+b((\couplestress_h,\sigma_h,\gamma_h),(u,\omega))\!+\!b((\couplestress,\sigma,\gamma),(u_h,\omega_h))\!=\!-f(u_h,\omega_h),
	\end{flalign*} 
	where $a$, $b$, and $f$ 
 are defined as in \eqref{eq:def_mcs_tdnns_gamma}.
	Further, 
    the following Galerkin ortho\-gonality is satisfied. Let $(u_h,\omega_h,\couplestress_h,\sigma_h,\gamma_h)\in \NedII^k\times\RT^{k-1}\times\MCS^{k-1}\times\HHJ^k\times\RT^{k-1}$ be the solution of Problem~\ref{prob:mcs_tdnns_method_gamma}. Then for all $(v_h,\xi_h,\Psi_h,\Theta_h,\delta_h)\in {\NedII}_{,0}^k\times\RT_{0}^{k-1}\times\MCS^{k-1}\times\HHJ^k\times\RT^{k-1}$ there holds
	\begin{flalign}
		\label{eq:galerkin_orthogonality}
\begin{split}
			a((\couplestress-\couplestress_h,\sigma-\sigma_h,\gamma-\gamma_h)&,(\Psi_h,\Theta_h,\delta_h))+b((\Psi_h,\Theta_h,\delta_h),(u-u_h,\omega-\omega_h)) \\
			&+ b((\couplestress-\couplestress_h,\sigma-\sigma_h,\gamma-\gamma_h),(v_h,\xi_h))=0.
\end{split}
	\end{flalign}
\end{theorem}
\begin{proof}
	We show consistency with respect to Problem~\ref{prob:mcs_tdnns_method_gamma}. The other cases are shown analogously (see also Corollary~\ref{cor:stability_convergence_mcs} for the required 
    adaptations for Problem~\ref{prob:mcs_method}).  Due to the regularity assumptions, there holds $[\![\couplestress_{nn}]\!]=[\![\sigma_{nt}]\!]=[\![u_{n}]\!]=[\![\omega_{t}]\!]=0$, and the duality pairings reduce to $L^2$ integrals, e.g.,
	\begin{flalign*}
		\langle \omega_h,\div \couplestress\rangle_{H(\div)^{\ast}} &\!=\! \sum_{T\in\mathcal{T}}\int_T \langle \omega_h,\div\couplestress\rangle\,dx \!-\! \sum_{F\in\mathcal{F}}\int_F\omega_{h,n}[\![\couplestress_{nn}]\!]\,ds \!=\! \int_{\Omega} \langle \omega_h,\div \couplestress\rangle\,dx,\\
		\langle \omega,\div \couplestress_h\rangle_{H(\div)^{\ast}}&\!=\! -\sum_{T\in\mathcal{T}}\int_T \langle \grad \omega,\couplestress_h\rangle dx \!+\! \sum_{F\in\mathcal{F}}\int_F\langle\couplestress_{h,nt},[\![\omega_{t}]\!]\rangle ds\!=\! -\int_{\Omega} \langle \grad \omega,\couplestress_h\rangle dx.
	\end{flalign*} 
 Using that for the exact solution 
 $\gamma = 2\mu_c(\half\curl u - \omega),$ and with integration by parts we get
 \begin{flalign*}
 -\int_{\Omega}\langle \half\curl u_h-\omega_h&,\gamma\rangle\,dx = -2\mu_c\int_{\Omega}\left\langle \half\curl u_h-\omega_h,\half\curl u-\omega\right\rangle\,dx \\
 &\!= 2\mu_c\int_{\Omega}\left(\left\langle \half\curl u-\omega,\omega_h\right\rangle-\half\left\langle\curl \left(\half\curl u-\omega\right),u_h\right\rangle\right)\,dx.
 \end{flalign*}
Using that 
for the exact solution $\couplestress = C_2( \grad \omega)$, $\sigma = \mathcal{C}(\grad u)$, and $\gamma = 2\mu_c(\half\curl u - \omega)$, we obtain together with the strong form \eqref{eq:strong_cosserat} of Cosserat elasticity
	\begin{flalign*}
		&a((\couplestress,\sigma,\gamma),(\couplestress_h,\sigma_h,\gamma_h))+b((\couplestress_h,\sigma_h,\gamma_h),(u,\omega))+b((\couplestress,\sigma,\gamma),(u_h,\omega_h))\\
  &\!=\! \int_{\Omega}\big(\langle \couplestress,\couplestress_h\rangle_{C^{-1}_2}\!+\!\langle \sigma,\sigma_h\rangle_{\mathcal{C}^{-1}}\!+\!\frac{1}{2\mu_c}\langle \gamma,\gamma_h\rangle\!+\!\langle\div\sigma,u_h\rangle \!+\! \langle \div\couplestress,\omega_h\rangle-\langle\sigma_h,\grad u\rangle\\
		&-\!\langle\couplestress_h,\grad \omega\rangle \!-\! \left\langle \half\curl u\!-\!\omega,\gamma_h\right\rangle\!+\!\mu_c\left\langle \curl u\!-\!2\omega,\omega_h\right\rangle\!-\!\mu_c\left\langle\curl \left(\half\curl u\!-\!\omega\right),u_h\right\rangle\big)dx\\
		&\!=\! \int_{\Omega}(\langle\div\sigma,u_h\rangle \!+\! \langle \div\couplestress,\omega_h\rangle\!+\!2\mu_c\left\langle \half\curl u\!-\!\omega,\omega_h\right\rangle\!-\!\mu_c\left\langle\curl \left(\half\curl u\!-\!\omega\right),u_h\right\rangle)\,dx  \\
    &= -f(u_h,\omega_h).
	\end{flalign*}
	We immediately 
    obtain the Galerkin orthogonality \eqref{eq:galerkin_orthogonality} by subtracting the equations.
\end{proof}

\subsection{Error analysis}
Before presenting the proof of convergence, we recall the definition and properties of the interpolation operators and norms used in the analysis of the TDNNS and MCS methods.
\begin{definition}[Canonical interpolation operators]
\label{def:canonical_interp}
We denote with $\mathcal{I}^{\NedII,k}:C(\Omega,\R^3)\to \NedII^k$, $\mathcal{I}^{\RT,k}:C(\Omega,\R^3)\to \RT^k$, $\mathcal{I}^{\MCS,k}:C(\Omega,\R^{3\times 3})\to \MCS^k$, and $\mathcal{I}^{\HHJ,k}:C(\Omega,\R^{3\times 3}_{\sym})\to \HHJ^k$ the canonical interpolation operators. $\mathcal{I}^{\RT,k}$ is defined by the following equations on $T\in\mathcal{T}$ and $F\in \mathcal{F}$
\begin{flalign*}
&\int_F(\mathcal{I}^{\RT,k}u-u)_n q\,ds = 0,\,\,  q\in \mathcal{P}^{k}(F),\quad \int_T\langle\mathcal{I}^{\RT,k}u-u, q\rangle\,dx = 0,\,\,  q\in [\mathcal{P}^{k-1}(T)]^3.
\end{flalign*}
For the definition of the other interpolation operators see, e.g., \cite{gopalakrishnan2020mass, pechstein2011tangential,Neun21}.
\end{definition}
\begin{lemma}
\label{lem:interpolation_ops_approximation}
	Let $\mathcal{I}^{\NedII,k}$, $\mathcal{I}^{\RT,k}$, $\mathcal{I}^{\MCS,k}$, and $\mathcal{I}^{\HHJ,k}$ be the canonical interpolation operators from Definition~\ref{def:canonical_interp}, and $1\leq l\leq k$. Assume that $u,\omega\in [H^1(\Omega)]^3\cap [H^{l+1}(\mathcal{T})]^3$, $\couplestress\in \{\tau \in [H^{l+1}(\mathcal{T})]^{3\times 3}\,:\, [\![\tau_{nt}]\!]=0\}$, and  $\sigma\in \{\tau \in [H^{l+1}(\mathcal{T})]^{3\times 3}_{\sym}\,:\, [\![\tau_{nn}]\!]=0\}$, where $H^{l+1}(\mathcal{T})=\Pi_{T\in\mathcal{T}}H^{l+1}(T)$. Then there hold for a constant $C>0$ independent of the mesh-size $h$ the following estimates
	\begin{flalign*}
			&\|u-\mathcal{I}^{\NedII,k}u\|_{V_h} \leq C h^l \|u\|_{H^{l+1}(\mathcal{T})}, \qquad\qquad \|\omega-\mathcal{I}^{\RT,k}\omega\|_{W_h} \leq C h^l \|\omega\|_{H^{l+1}(\mathcal{T})},\\
   &\|\sigma-\mathcal{I}^{\HHJ,k}\sigma\|_{L^2}+\sqrt{\sum_{F\in\mathcal{F}}h\|(\sigma-\mathcal{I}^{\HHJ,k}\sigma)_{nn}\|^2_{L^2(F)}}\leq Ch^{l+1}\|\sigma\|_{H^{l+1}(\mathcal{T})},\\
			& \|\couplestress-\mathcal{I}^{\MCS,k}\couplestress\|_{L^2}+\sqrt{\sum_{F\in\mathcal{F}}h\|(\couplestress-\mathcal{I}^{\MCS,k}\couplestress)_{nt}\|^2_{L^2(F)}}\leq Ch^{l+1}\|\couplestress\|_{H^{l+1}(\mathcal{T})},
		\end{flalign*}
where $\|\cdot\|_{V_h}$ and $\|\cdot\|_{W_h}$ are defined as in Theorem~\ref{thm:robustness_tdnns}.
\end{lemma}
\begin{proof}
	For a proof we refer to \cite[Theorem 4.21 and Theorem 4.23]{Sin2009} and \cite[Theorem~5.8, (6.13) ]{gopalakrishnan2020mass0} (noting the slightly different definition of the $\MCS$ space \eqref{eq:mcs_fes}).
\end{proof}

\begin{lemma}
\label{lem:orthogonality_RT_int}
There holds for $u\in [H^1(\Omega)]^3$ the following orthogonality property
\begin{flalign*}
\langle \div\tau_h, u-\mathcal{I}^{\RT,k}u\rangle_{H(\div)^{\ast}}=0,\qquad \forall \tau_h\in \MCS^k.
\end{flalign*}
\end{lemma}
\begin{proof}
Writing out the duality pairing \eqref{eq:duality_RT_MCS} and noting that $\div \!\tau_h|_T\!\in\! [\mathcal{P}^{k-1}(T)]^3$ and $\tau_{h,nn}\in \mathcal{P}^k(F)$ the claim follows by the definition of the RT interpolation operator
\begin{flalign*}
&\langle \div\tau_h, u-\mathcal{I}^{\RT,k}u\rangle_{H(\div)^{\ast}} \\
&\qquad= \sum_{T\in\mathcal{T}}\int_T \langle \div \tau_h,u-\mathcal{I}^{\RT,k}u\rangle\,dx - \sum_{F\in\mathcal{F}}\int_F\tau_{h,nn}(u-\mathcal{I}^{\RT,k}u)_n\,ds = 0.
\end{flalign*}
\end{proof}

\begin{lemma}
	\label{lem:equivalency_norms}
	Let $k\geq 1$ be an integer, $\sigma_h\in\HHJ^{k}$, and $\couplestress_h\in \MCS^{k}$. Then, there hold the following norm equivalencies, $a\simeq b$ if $a\leq C\,b$ and $b\leq C\,a$ for a constant $C>0$ independent of the mesh-size,
	\begin{flalign*}
			&\|\sigma_h\|_{L^2}^2\simeq \|\sigma_h\|_{L^2}^2 + h\sum_{F\in\mathcal{F}}\|\sigma_{h,nn}\|^2_{L^2(F)},\quad \|\couplestress_h\|_{L^2}^2\simeq \|\couplestress_h\|_{L^2}^2 + h\sum_{F\in\mathcal{F}}\|\couplestress_{h,nt}\|^2_{L^2(F)} .
	\end{flalign*}
\end{lemma}
\begin{proof}
	Scaling arguments, see also \cite{pechstein2011tangential, gopalakrishnan2020mass}.
\end{proof}
Now we are ready  
to prove convergence rates robust in $\mu_c$.
\begin{theorem}[Convergence]
	\label{thm:convergence}
	Let $(u,\omega,\couplestress,\sigma)$ be the exact solution of linear Cosserat elasticity and $(u_h,\omega_h,\couplestress_h,\sigma_h,\gamma_h)\in \NedII^k\times\RT^{k-1}\times\MCS^{k-1}\times\HHJ^k\times\RT^{k-1}$ the discrete solution of Problem~\ref{prob:mcs_tdnns_method_gamma}. Assume for $0\leq l\leq k-1$ the regularity $u\in [H^1(\Omega)]^3\cap [H^{l+2}(\mathcal{T})]^3$, $\omega\in [H^1(\Omega)]^3\cap [H^{l+1}(\mathcal{T})]^3$, $\couplestress\in [H^1(\Omega)]^{3\times 3}\cap [H^{l+1}(\mathcal{T})]^{3\times 3}$, and $\sigma\in [H^1(\Omega)]^{3\times 3}\cap [H^{l+1}(\mathcal{T})]^{3\times 3}$. Then there holds, with a constant $C\neq C(\mu_c)$,
	\begin{flalign}
		 \begin{split}
   \label{eq:convergence}
			\|u-&u_h\|_{V_h} + \|\omega-\omega_h\|_{W_h} + \|\couplestress-\couplestress_h\|_{L^2} + \|\sigma-\sigma_h\|_{L^2} + \|\gamma-\gamma_h\|_{\Gamma} \\
			 &\qquad\leq C h^l\left(\|u\|_{H^{l+1}} +\|\omega\|_{H^{l+1}} + \|\couplestress\|_{H^{l}} + \|\sigma\|_{H^{l}}+ \frac{1}{\sqrt{\mu_c}}\|\gamma\|_{H^{l}}\right),
		 \end{split}\\
   \begin{split}
   \label{eq:improved_convergence}
			\|u-&u_h\|_{V_h} + \|\omega_h-\mathcal{I}^{\RT,k-1}\omega\|_{W_h} + \|\couplestress-\couplestress_h\|_{L^2} + \|\sigma-\sigma_h\|_{L^2} + \|\gamma-\gamma_h\|_{\Gamma} \\
			 &\qquad\leq C h^{l+1}\left(\|u\|_{H^{l+2}} + \|\couplestress\|_{H^{l+1}} + \|\sigma\|_{H^{l+1}}+ \frac{1}{\sqrt{\mu_c}}\|\gamma\|_{H^{l+1}}\right),
		 \end{split}
	\end{flalign}
	where the norms $\|\cdot\|_{\Gamma}$, $\|\cdot\|_{V_h}$, and $\|\cdot\|_{W_h}$ are defined as in Theorem~\ref{thm:robustness_tdnns}.
\end{theorem}
\begin{proof}
	Throughout the proof, we use the notation $a\lesssim b$ if $a\leq C\, b$ with a constant $C>0$ independent of the mesh-size and $\mu_c$. We add and subtract the canonical interpolation operators from Definition~\ref{def:canonical_interp}, apply the triangle inequality, and use the notation $\hat{u}_h:=u_h-\mathcal{I}^{\NedII,k}u$, $\hat{u}:=u-\mathcal{I}^{\NedII,k}u$, etc.,
	\begin{flalign*}
		&\|u-u_h\|_{V_h} + \|\omega-\omega_h\|_{W_h} + \|\couplestress-\couplestress_h\|_{L^2} + \|\sigma-\sigma_h\|_{L^2} + \|\gamma-\gamma_h\|_{\Gamma}\leq \\
		&
		\|\hat{u}\|_{V_h}\! +\! \|\hat\omega\|_{W_h} \!+ \!\|\hat\couplestress\|_{L^2} \!+ \!\|\hat\sigma\|_{L^2} \!+ \!\|\hat\gamma\|_{\Gamma} \!+ \!\|\hat{u}_h\|_{V_h} \!+\! \|\hat\omega_h\|_{W_h} \!+ \!\|\hat\couplestress_h\|_{L^2}\! +\! \|\hat\sigma_h\|_{L^2} \!+\! \|\hat\gamma_h\|_{\Gamma}.
	\end{flalign*}
	The first five terms can be estimated directly with Lemma~\ref{lem:interpolation_ops_approximation}
 \begin{flalign*}
 \|\hat{u}\|_{V_h} + \|\hat\omega\|_{W_h} &+ \|\hat\couplestress\|_{L^2} + \|\hat\sigma\|_{L^2} + \|\hat\gamma\|_{\Gamma}\\
 &\lesssim h^l\left(\|u\|_{H^{l+1}} +\|\omega\|_{H^{l+1}} + \|\couplestress\|_{H^{l}} + \|\sigma\|_{H^{l}}+ \nicefrac{1}{\sqrt{\mu_c}}\|\gamma\|_{H^{l}}\right).
 \end{flalign*}
 For the remaining terms, we use the inf-sup stability of Problem~\ref{prob:mcs_tdnns_method_gamma} and the Galerkin orthogonality \eqref{eq:galerkin_orthogonality}. With  $\|\cdot\|_{X_h}$ and $\|\cdot\|_{Y_h}$ as in the proof of Theorem~\ref{thm:robustness_tdnns}, $\varUpsilon_h:= (\Psi_h,\Theta_h,\delta_h)\in X_h$, and $\kappa_h:= (v_h,\xi_h)\in Y_h$ we obtain 
	\begin{flalign*}
		&\|\hat{u}_h\|_{V_h} + \|\hat{\omega}_h\|_{W_h} + \|\hat\couplestress_h\|_{L^2} + \|\hat\sigma_h\|_{L^2} + \|\hat\gamma_h\|_{\Gamma}\leq \| (\hat\couplestress_h,\hat\sigma_h,\hat\gamma_h)\|_{X_h} + \|(\hat u_h,\hat\omega_h)\|_{Y_h}\\
		&\lesssim \sup\limits_{(\varUpsilon_h,\kappa_h)\in {X_h\times Y_h}}\frac{a((\hat\couplestress_h,\hat\sigma_h,\hat\gamma_h),\varUpsilon_h)+ b((\hat\couplestress_h,\hat\sigma_h,\hat\gamma_h),\kappa_h)+b(\varUpsilon_h,(\hat u_h,\hat\omega_h))}{\| \varUpsilon_h\|_{X_h} + \|\kappa_h\|_{Y_h}}\\
		& = \sup\limits_{(\varUpsilon_h,\kappa_h)\in {X_h\times Y_h}}\frac{a((\hat\couplestress,\hat\sigma,\hat\gamma),\varUpsilon_h)+ b((\hat\couplestress,\hat\sigma,\hat\gamma),\kappa_h)+b(\varUpsilon_h,(\hat u,\hat\omega))}{\| \varUpsilon_h\|_{X_h} + \|\kappa_h\|_{Y_h}}.
	\end{flalign*}
We estimate the bilinear forms $a$ and $b$. 
 For $a$ we get with Cauchy-Schwarz inequality
	\begin{flalign*}
		|a((\hat\couplestress,\hat\sigma,\hat\gamma),\varUpsilon_h)|
		&\lesssim \|(\hat \couplestress, \hat \sigma, \hat \gamma)\|_{X_h}\|\varUpsilon_h\|_{X_h}.
	\end{flalign*}
	For the two expressions involving $b$ we carefully treat the element boundary terms
	\begin{flalign*}
		|b((\hat\couplestress,\hat\sigma,\hat\gamma),\kappa_h)|
		&\!\leq \!\left|\sum_{T\in\mathcal{T}}-\int_T(\langle\hat{\couplestress},\grad \xi_h\rangle + \langle\hat{\sigma},\grad v_h\rangle)\,dx \!+\! \int_{\partial T}(\langle \hat{\couplestress}_{nt},\xi_{h,t}\rangle \!+\! \hat\sigma_{nn}v_{h,n})\,ds\right| \\
		&\qquad+ \|\hat{\gamma}\|_{\Gamma}\sqrt{\mu_c}\left\|\half\curl v_h-\xi_h\right\|_{L^2}.
	\end{flalign*}
	Noting that $\hat{m}$ is tangential-normal continuous, we can estimate with Cauchy-Schwarz
	\begin{flalign*}
		&\sum_{T\in\mathcal{T}}\Big(\int_T \langle\hat{\couplestress},\grad \xi_h\rangle dx  \!-\! \int_{\partial T}\langle \hat{\couplestress}_{nt},\xi_{h,t}\rangle ds\Big)\!=\!\sum_{T\in\mathcal{T}}\int_T \langle\hat{\couplestress},\grad \xi_h\rangle dx  \!-\!\sum_{F\in\mathcal{F}} \int_{F}\langle \hat{\couplestress}_{nt},[\![\xi_{h,t}]\!]\rangle ds\\
		&\qquad\qquad\leq \sum_{T\in\mathcal{T}}\|\hat{\couplestress}\|_{L^2(T)}\|\grad \xi_h\|_{L^2(T)} +\sum_{F\in\mathcal{F}}\sqrt{h}\|\hat{\couplestress}_{nt}\|_{L^2(F)}\frac{1}{\sqrt{h}}\|[\![\xi_{h,t}]\!]\|_{L^2(F)}\\
		&\qquad\qquad\leq \|\hat{\couplestress}\|_{L^2}\sqrt{\sum_{T\in\mathcal{T}}\|\grad \xi_h\|^2_{L^2(T)}} +\sqrt{\sum_{F\in\mathcal{F}}h\|\hat{\couplestress}_{nt}\|^2_{L^2(F)}}\sqrt{\sum_{F\in\mathcal{F}}\frac{1}{h}\|[\![\xi_{h,t}]\!]\|^2_{L^2(F)}}\\
		&\qquad\qquad\lesssim \Big(\|\hat{\couplestress}\|_{L^2}+\sqrt{\sum_{F\in\mathcal{F}}h\|\hat{\couplestress}_{nt}\|^2_{L^2(F)}}\Big) \|\xi_h\|_{W_h}.
	\end{flalign*}
    Similarly, since 
    $\hat{\sigma}$ is normal-normal continuous and symmetric we have
	\begin{flalign*}
		\Big|\sum_{T\in\mathcal{T}}\Big(\int_T \langle\hat{\sigma},\sym\grad v_h\rangle\,dx \!-\! \int_{\partial T} \hat\sigma_{nn}v_{h,n}\,ds\Big)\Big| \!\leq\! \Big(\|\hat{\sigma}\|_{L^2}\!+\!\sqrt{\sum_{F\in\mathcal{F}}h\|\hat{\sigma}_{nn}\|^2_{L^2(F)}}\Big) \|v_h\|_{V_h}.
	\end{flalign*}
	Thus, we obtain
	\begin{flalign*}
		b((\hat\couplestress,\hat\sigma,\hat\gamma),\kappa_h)\!\lesssim\! \Big(\|(\hat{\couplestress},\hat{\sigma},\hat{\gamma})\|_{X_h}\!+\!\sqrt{\sum_{F\in\mathcal{F}}h\|\hat{\couplestress}_{nt}\|^2_{L^2(F)}}\!+\!\sqrt{\sum_{F\in\mathcal{F}}h\|\hat{\sigma}_{nn}\|^2_{L^2(F)}}\Big)\|\kappa_h\|_{Y_h}.
	\end{flalign*}
For the second expression involving $b$ we have
 \begin{flalign*}
 b(\varUpsilon_h,(\hat u,\hat\omega)) &= \sum_{T\in\mathcal{T}}\int_T(\langle\div\Psi_h,\hat\omega\rangle - \langle\Theta_h,\grad \hat u\rangle)\,dx - \int_{\partial T}(\Psi_{h,nn}\hat\omega_n - \Theta_{h,nn}\hat u_n)\,ds\\
		&\qquad-\int_{\Omega}\left\langle\delta_h,\half\curl \hat u-\hat\omega\right\rangle\,dx.
 \end{flalign*}
 Using Lemma~\ref{lem:orthogonality_RT_int} the terms involving $\Psi_h$ vanish.  The terms with $\Theta_h$ can be estimated in the same vein as before, using the equivalence of discrete norms in Lemma~\ref{lem:equivalency_norms},
 \begin{flalign*}
 \left|\sum_{T\in\mathcal{T}}-\int_T \langle\Theta_h,\grad \hat u\rangle\,dx + \int_{\partial T} \Theta_{h,nn}\hat u_n\,ds\right|\lesssim \|\Theta_h\|_{L^2}\|\hat u\|_{V_h}.
 \end{flalign*}
	For the last term there holds with the commuting property of interpolants, \\$\curl(\mathcal{I}^{\NedII^k}u) = \mathcal{I}^{\RT^{k-1}}(\curl u)$,
 \begin{flalign*}
 \left|\int_{\Omega}\left\langle\delta_h,\half\curl \hat u-\hat\omega\right\rangle\,dx\right|=\left|\int_{\Omega}\left\langle\delta_h,\nicefrac{1}{\mu_c}\hat\gamma\right\rangle\,dx\right|\le \|\delta_h\|_{\Gamma}\|\hat\gamma\|_{\Gamma}.
 \end{flalign*}
 Thus, we arrive at
 \begin{flalign*}
 b(\varUpsilon_h,(\hat u,\hat\omega))\!\lesssim \!\|\Theta_h\|_{L^2}\|\hat u\|_{V_h}\!+\!\|\delta_h\|_{\Gamma}\|\hat\gamma\|_{\Gamma}\!\lesssim \!(\|\hat u\|_{V_h}\!+\!\|\hat\gamma\|_{\Gamma})\|(\Theta_h,\Psi_h,\delta_h)\|_{X_h}.
 \end{flalign*}
 Recognizing that the terms involving $\varUpsilon_h$ and $\kappa_h$ in the numerator and denominator of the supreme cancel, we obtain with Lemma~\ref{lem:interpolation_ops_approximation} for $0\leq l\leq k-1$
	\begin{flalign}
		\| (\hat\couplestress_h,\hat\sigma_h&,\hat\gamma_h)\|_{X_h} +  \|(\hat u_h,\hat\omega_h)\|_{Y_h}\nonumber\\
		&\!\lesssim\!\| \hat\couplestress\|_{L^2}\!+\!\|\hat\sigma\|_{L^2} \!+\! \|\hat\gamma\|_{\Gamma} \!+\! \sqrt{\sum_{F\in\mathcal{F}}h\|\hat{\couplestress}_{nt}\|^2_{L^2(F)}}\!+\!\sqrt{\sum_{F\in\mathcal{F}}h\|\hat{\sigma}_{nn}\|^2_{L^2(F)}} \!+\!  \|\hat u\|_{V_h} \nonumber\\
		&\lesssim h^{l}\left( \|\couplestress\|_{H^{l}} + \|\sigma\|_{H^{l}} + \|u\|_{H^{l+1}} + \nicefrac{1}{\sqrt{\mu_c}}\|\gamma\|_{H^{l}}\right),\label{eq:conv_intermediate}
	\end{flalign}
	which yields \eqref{eq:convergence}. For \eqref{eq:improved_convergence} we note that estimate \eqref{eq:conv_intermediate} can be improved to $h^{l+1}$ convergence as it is independent of $\omega_h$.
\end{proof}

\begin{remark}
	For $k=1$, we do not obtain convergence for \eqref{eq:convergence}. The reason is that for the rotations $\omega_h$ the Raviart--Thomas space $\RT^0$ does not contain the whole linear polynomial space. However, thanks to the orthogonality Lemma~\ref{lem:orthogonality_RT_int}, the filtered error $\omega_h-\IntRT[k-1]$ leads to super-convergence of the other fields.
\end{remark}

\begin{corollary}[Stability \& convergence of Problem~\ref{prob:mcs_method}]
	\label{cor:stability_convergence_mcs}
	Adopt the assumptions from Theorem~\ref{thm:convergence}. Problem~\ref{prob:mcs_method} is stable, consistent, and there hold the estimates
 \begin{flalign*}
			\|u-&u_h\|_{H^1} + \|\omega-\omega_h\|_{W_h} + \|\couplestress-\couplestress_h\|_{L^2} + \|\gamma-\gamma_h\|_{\Gamma} \\
			 &\qquad\leq C h^l\left(\|u\|_{H^{l+1}} +\|\omega\|_{H^{l+1}} + \|\couplestress\|_{H^{l}} + \frac{1}{\sqrt{\mu_c}}\|\gamma\|_{H^{l}}\right),\\
			\|u-&u_h\|_{H^1} + \|\omega_h-\mathcal{I}^{\RT,k-1}\omega\|_{W_h} + \|\couplestress-\couplestress_h\|_{L^2} + \|\gamma-\gamma_h\|_{\Gamma} \\
			 &\qquad\leq C h^{l+1}\left(\|u\|_{H^{l+2}} + \|\couplestress\|_{H^{l+1}} + \frac{1}{\sqrt{\mu_c}}\|\gamma\|_{H^{l+1}}\right).
	\end{flalign*}
\end{corollary}
\begin{proof}
	Rewrite the elasticity part $u$ in terms of an equivalent mixed method with $(u,\sigma)\in [H^1(\Omega)]^3\times [L^2(\Omega)]^{3\times 3}_{\sym}$ and $(u_h,\sigma_h)\in \vLag^k\times [\mathcal{P}^{k-1}(\mathcal{T})]^{3\times 3}_{\sym}$. 
 The proof of Theorems~\ref{thm:robustness_tdnns}, \ref{thm:consistency}, and \ref{thm:convergence} can readily be adapted by defining the norm $\|u_h\|_{V_h}=\|u_h\|_{H^1}$ and noting that $\langle \div \sigma_h,u_h\rangle_{H(\curl)^{\ast}}=-\int_{\Omega}\langle\sigma_h,\grad u_h\rangle\,dx$.
\end{proof}

\subsection{Post-processing of rotations}
From Theorem~\ref{thm:convergence} and Corollary~\ref{cor:stability_convergence_mcs} we obtain that the rotation field $\omega_h$ converges with order $k-1$ in the $W_h$-norm, whereas the couple stress $\couplestress_h$ converges with order $k$ in the $L^2$-norm. By the relation $\grad\omega =C_2^{-1}(\couplestress)$, however, we can construct a rotation field $\tilde{\omega}_h$ with improved convergence rates. A standard approach is to solve the problem
\begin{flalign*}
\tilde\omega_h=\argmin_{\substack{w_h\in [\mathcal{P}^k(\mathcal{T})]^3\\
\int_{T}w_h\,dx =\int_{T}\omega_h\,dx\,,\forall T\in\mathcal{T}}} \|\grad w_h-C_2^{-1}(\couplestress_h)\|_{L^2}.
\end{flalign*}
It is well-known, see e.g. \cite{Ste1991,CGS2010}, that $\tilde{\omega}_h$ converges in the $L^2$-norm with one order more than $\couplestress_h$. However, we  
cannot expect an improvement in the $W_h$-norm.

Motivated by \cite{CHZ2023}, where a N\'ed\'elec function is post-processed such that improved convergence in an $H(\curl)$-type norm is achieved, we propose a second type of post-processing. To this end, we define e.g. $\BDM^{k,\mathrm{dc}}$ as the BDM space of order $k$, where the normal continuity is broken, i.e., the coupling degrees of freedom are doubled. Further, we denote with $\RT^{0,\mathrm{dc},*}$ the discrete elementwise dual space, whose degrees of freedom coincide with the canonical interpolant from Definition~\ref{def:canonical_interp}.
\begin{problem}[Post-processing]
\label{prob:postprocess2}
Let $(\omega_h,\couplestress_h)\in \RT^{k-1}\times \MCS^{k-1}$ be part of the solution of Problem~\ref{prob:mcs_method} or Problem~\ref{prob:mcs_tdnns_method}. Find $(\tilde{\omega}_h,\chi_h)\in \BDM^{k,\mathrm{dc}}\times\RT^{0,\mathrm{dc},*}$ such that for all $({\xi}_h,\rho_h)\in \BDM^{k,\mathrm{dc}}\times\RT^{0,\mathrm{dc},*}$
\begin{subequations}
\label{eq:postprocess2}
		\begin{alignat}{3}
			&\sum_{T\in\mathcal{T}}\int_{T}\langle\grad \tilde{\omega}_h,\grad\xi_h\rangle\,dx&+&\chi_h(\xi_h) &=& \sum_{T\in\mathcal{T}}\int_{T}\langle C_2^{-1}( \couplestress_h),\grad\xi_h\rangle\,dx,\label{eq:postprocess2_a}\\
			&\rho_h(\tilde{\omega}_h)& & &=& \rho_h(\omega_h).\label{eq:postprocess2_b}
		\end{alignat}
\end{subequations}
\end{problem}
Problem~\ref{prob:postprocess2} can be solved elementwise. Equation \eqref{eq:postprocess2_b} states that $\IntRT[0]\tilde{\omega}_h=\IntRT[0]\omega_h$. Only the additional shape functions required 
to extend $\RT^{0}$ to $\BDM^k$ are optimized in order to satisfy the equation $\grad\tilde\omega_h =C_2^{-1}(\couplestress_h)$ in an $L^2$ sense. For this post-processing, which we  
use in all numerical examples, 
we also observed an increased convergence rate in the $W_h$-norm. A rigorous analysis of this post-processing scheme is, however, out of the scope of this work and is a topic of future research.

\section{Numerical examples}
\label{sec:numerics}

All numerical examples are performed in the open-source finite element library NGSolve\footnote{\href{www.ngsolve.org}{www.ngsolve.org}} \cite{Sch97,Sch14}.

\begin{table}[ht!]
\small
    \centering
    \begin{tabular}{c|c|c}
    method & description & reference\\
    \hline
        $\Lag^k$ &  Lagrangian elements for $u$ and $\omega$ & Problem~\ref{prob:primal_lagrange_formulation} \\
        M$^{k}$ & Lagrangian elements for $u$ and MCS for $\omega$ & Problem~\ref{prob:mcs_method} \\
        M-T$^{k}$ & TDNNS for $u$ and MCS for $\omega$ & Problem~\ref{prob:mcs_tdnns_method} \\
    \end{tabular}
    \caption{Used methods of polynomial order $k$ for numerical examples.}
    \label{tab:cosserat_elements}
\end{table}
The different methods for discretizing Cosserat elasticity are displayed in Table~\ref{tab:cosserat_elements}. 
We compute the relative errors in the norms defined within Theorem~\ref{thm:convergence} and Corollary~\ref{cor:stability_convergence_mcs}, if not stated otherwise, and the respective estimated order of convergence (eoc). For comparison, the number of coupling degrees of freedom (ncdof) is considered. The displacement, rotation, elasticity stress, and couple stress errors read
\begin{flalign*}
u_{\mathrm{err}}=\frac{\|u_h-u\|_{V_h}}{\|u\|_{H^1}}, \,\,\omega_{\mathrm{err}}=\frac{\|\omega_h-\omega\|_{W_h}}{\|\omega\|_{H^1}},\,\, \sigma_{\mathrm{err}}=\frac{\|\sigma_h-\sigma\|_{L^2}}{\|\sigma\|_{L^2}}, \,\,\couplestress_{\mathrm{err}}=\frac{\|\couplestress_h-\couplestress\|_{L^2}}{\|\couplestress\|_{L^2}}.
\end{flalign*}
When we use the MCS method to discretize the rotation field $\omega$, we additionally compute the error of the post-processed $\tilde{\omega}_h$ from Problem~\ref{prob:postprocess2}.

\subsection{Cylindrical bending of plate}
\label{subsec:cylindrical_bending_plate}
The first example we consider is a cylindrical bending of a thick plate \cite{XHZH2019} with length $L=20$, height $H=2$, and thickness $t=20$, see Figure~\ref{fig:cylindrical_bending_plate}. The material parameters\footnote{In \cite{XHZH2019} the authors use the parameters $a,b,c$, which are related by $\alpha=4\mu c L_c^2$, $\beta=4\mu b L_c^2$, $\gamma=4\mu L_c^2$, $\mu_c = \mu a$.} are chosen as $E=2500$, $\nu\in\{0.25,0.4999\}$ (and corresponding Lam\'e parameters $\mu=E/(2(1+\nu))$, $\lambda= E\nu/((1+\nu)(1-2\nu))$), $\mu_c=0.5\mu$, $\alpha=2\mu\,L_c^2$, $\beta=2\mu\,L_c^2$, and $\gamma=4\mu\,L_c^2$, with $L_c=1$ the characteristic length. A bending moment $M_x=100$ is applied on the left and right boundary. To ensure a pure bending, couple stresses are  
imposed at the front and back boundary to compensate for the effect of the Poisson ratio. The exact solution is given by
\begin{flalign*}
& u_x \!=\! \frac{M_xxy}{D+\gamma\,H},\quad u_y \!=\! -\frac{M_x}{2(D\!+\!\gamma\,H)}\Big(x^2\!+\!\frac{\nu}{1\!-\!\nu}y^2\Big)+\frac{M_x}{24(D\!+\!\gamma\,H)}\Big(L^2+\frac{\nu}{1\!-\!\nu}H^2\Big),\\
& u_z=0,\quad \omega_x=\omega_y=0,\qquad\omega_z = -\frac{M_xx}{D+\gamma\,H},	
\end{flalign*}
where $D=\frac{E\,H^3}{12(1-\nu^2)}$. No body forces apply. The non-zero stress components are
\begin{flalign*}
	\sigma_{xx} \!=\! \frac{E}{1-\nu^2}\frac{M_xy}{D\!+\!\gamma\,H},\quad \sigma_{zz} \!=\! \frac{\nu\, E}{1\!-\!\nu^2}\frac{M_xy}{D\!+\!\gamma\,H},\quad \couplestress_{xz} \!=\! -\frac{\beta M_x}{D\!+\!\gamma\,H},\quad \couplestress_{zx} \!=\! -\frac{\gamma M_x}{D\!+\!\gamma\,H}.
\end{flalign*}
Due to symmetry, we mesh only one quarter of the domain  
and prescribe  
symmetry boundary conditions.  
The other boundaries are left free  
or the bending forces are applied as Neumann boundary conditions. 
To obtain a well-posed problem, the mean value of $u_y$ is constrained to zero in terms of a Lagrange multiplier, noting that for the exact solution, there holds $\int_{\Omega}u_y\,dx=0$. We further note that for quadratic elements all methods reproduce the exact solution exactly, demonstrating 
the consistency of the proposed methods. Thus, we consider only the lowest-order linear elements for the displacements in this benchmark.

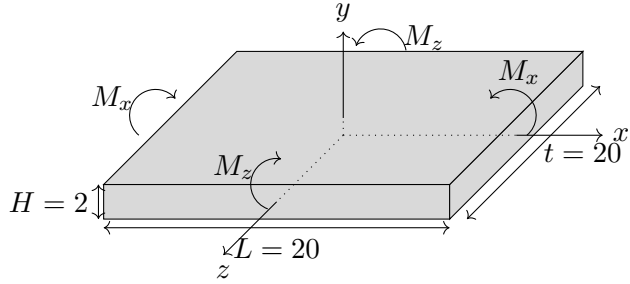
\begin{figure}[ht!]
	\centering
	\begin{tikzpicture}[scale=0.22]
		\def\L{20} 
		\def\H{2} 
		\def\t{20}

		\draw[fill=gray!30] (0,0,0) -- ++(\L,0,0) -- ++(0,\H,0) -- ++(-\L,0,0) -- cycle;
		\draw[fill=gray!30] (\L,0,0) -- ++(0,0,-\t) -- ++(0,\H,0) -- ++(0,0,\t) -- cycle;
		\draw[fill=gray!30] (0,\H,0) -- ++(\L,0,0) -- ++(0,0,-\t) -- ++(-\L,0,0) -- cycle;

		\draw[<->] (\L,-0.5,0) -- (0,-0.5,0) node[midway, below] {$L=20$};
		\draw[<->] (\L*1.05,0,0) -- (\L*1.05,0,-\t) node[midway, right] {$t=20$};
		\draw[<->] (-0.3,0,0) -- (-0.3,\H,0) node[midway, left] {$H=2$};

		\centerarc[->](3,6)(230:40:1.5)
		\centerarc[->](23.5,6)(-50:140:1.5)
		\centerarc[->](10,2)(250:80:1.5)
		\centerarc[->](16,9.5)(10:160:1.5)

		\draw[-, dotted] (\L/2,\H/2,-\t/2) -- ++(\L/2,0,0);
		\draw[->] (\L,\H/2,-\t/2)-- ++(5,0,0) node[right] {$x$};
		\draw[-, dotted] (\L/2,\H/2,-\t/2) -- ++(0,\H/2,0);
		\draw[->] (\L/2,\H,-\t/2)-- ++(0,5,0) node[above] {$y$};
		\draw[-, dotted] (\L/2,\H/2,-\t/2) -- ++(0,0,\t/2);
		\draw[->] (\L/2,\H/2,0)-- ++(0,0,8)  node[below] {$z$};

		\draw ({0.5},{7}) node {$M_x$};
		\draw ({24},{8.5}) node {$M_x$};
		\draw ({7.5},{3}) node {$M_z$};
		\draw ({18.5},{10.5}) node {$M_z$};
	\end{tikzpicture}
 \vspace*{-0.6cm}
	\caption{Geometry of cylindrical bending of a plate example.}
	\label{fig:cylindrical_bending_plate}
\end{figure}

Poisson ratios $\nu\in\{0.25,0.4999\}$ are considered to test for robustness in the nearly incompressible regime, where $\nu=0.4999$ corresponds to an almost incompressible material. As displayed in Table~\ref{tab:res_cyl_bend_plate_compr} all proposed methods perform equally well for $\nu=0.25$. Although the theory indicates only linear convergence for the fields (except for $\omega$ in the MCS settings), we observe higher convergence rates, which might results from  
the low polynomial exact solution. For more complicated solutions, this behavior vanishes. The MCS-TDNNS method is superior to the MCS method. This might be due to the plate structure with one dimension being smaller than the others and the TDNNS method being more robust with respect to anisotropic structures compared to standard elasticity discretization by Lagrangian elements. In the nearly incompressible regime, $\nu=0.4999$, we clearly observe reduced convergence for coarse meshes, cf. Table~\ref{tab:res_cyl_bend_plate_incompr}, except for the MCS-TDNNS method, highlighting the  
robustness of the TDNNS method for nearly incompressible elasticity. Specifically, the stress $\sigma_h$ can only be approximated well with the TDNNS method.

\begin{table}
  \centering
  \resizebox{0.9\textwidth}{!}{
\begin{tabular}{c|c|cc|cc|cc|cc|cc}
& ncdof & $u_{\mathrm{err}}$ & eoc & $\omega_{\mathrm{err}}$ & eoc & $\tilde{\omega}_{\mathrm{err}}$ & eoc & $\sigma_{\mathrm{err}}$ & eoc & $\couplestress_{\mathrm{err}}$ & eoc \\ 
 \hline
$\Lag^1$   & 139 & 0.4913 & -  & 0.4519 & -   & - & -  & 1.0522 & -  & 0.3960 & -  \\
  & 705 & 0.2549 & 0.95 & 0.2384 & 0.92  & - & -  & 0.8688 & 0.28 & 0.2080 & 0.93 \\
  & 4324 & 0.0861 & 1.56 & 0.0775 & 1.62  & - & -  & 0.5306 & 0.71 & 0.0660 & 1.66 \\
  & 29910 & 0.0262 & 1.72 & 0.0207 & 1.91  & - & -  & 0.2801 & 0.92 & 0.0175 & 1.92 \\
\hline
M$^1$   & 432 & 0.4655 & -  & 0.4812 & -  & 0.4181 & -  & 1.0543 & -  & 0.3452 & -  \\
  & 2991 & 0.2385 & 0.96 & 0.3282 & 0.55 & 0.2208 & 0.92 & 0.8697 & 0.28 & 0.1916 & 0.85 \\
  & 22296 & 0.0797 & 1.58 & 0.2672 & 0.30 & 0.0705 & 1.65 & 0.5281 & 0.72 & 0.0607 & 1.66 \\
  & 172254 & 0.0246 & 1.70 & 0.2636 & 0.02 & 0.0187 & 1.92 & 0.2779 & 0.93 & 0.0160 & 1.92 \\
\hline
M-T$^1$   & 928 & 0.2240 & -  & 0.3197 & -  & 0.1910 & -  & 0.7787 & -  & 0.1684 & -  \\
  & 6543 & 0.1196 & 0.91 & 0.2800 & 0.19 & 0.1052 & 0.86 & 0.6093 & 0.35 & 0.0974 & 0.79 \\
  & 49037 & 0.0470 & 1.35 & 0.2653 & 0.08 & 0.0366 & 1.52 & 0.3643 & 0.74 & 0.0328 & 1.57 \\
  & 379481 & 0.0179 & 1.40 & 0.2643 & 0.01 & 0.0105 & 1.80 & 0.1958 & 0.90 & 0.0093 & 1.82
\end{tabular}}
\caption{Results of relative error of cylindrical bending of a plate with $\nu=0.25$.}
\label{tab:res_cyl_bend_plate_compr}
\end{table}

\begin{table}
  \centering
  \resizebox{0.9\textwidth}{!}{
\begin{tabular}{c|c|cc|cc|cc|cc|cc}
& ncdof & $u_{\mathrm{err}}$ & eoc & $\omega_{\mathrm{err}}$ & eoc & $\tilde{\omega}_{\mathrm{err}}$ & eoc & $\sigma_{\mathrm{err}}$ & eoc & $\couplestress_{\mathrm{err}}$ & eoc \\
 \hline
$\Lag^1$  & 139 & 0.6310 & -  & 0.5821 & -   & - & -  & 12.256 & -  & 0.5288 & -  \\
 & 705 & 0.4979 & 0.34 & 0.4563 & 0.35  & - & -  & 14.577 & -0.25 & 0.3943 & 0.42 \\
 & 4324 & 0.2753 & 0.85 & 0.2547 & 0.84  & - & -  & 17.881 & -0.29 & 0.2198 & 0.84 \\
 & 29910 & 0.1080 & 1.35 & 0.0970 & 1.39  & - & -  & 13.361 & 0.42 & 0.0827 & 1.41 \\
\hline
M$^{1}$  & 432 & 0.6228 & -  & 0.6084 & -  & 0.5660 & -  & 12.159 & -  & 0.4877 & -  \\
 & 2991 & 0.4938 & 0.34 & 0.4981 & 0.29 & 0.4511 & 0.33 & 14.586 & -0.26 & 0.3858 & 0.34 \\
 & 22296 & 0.2729 & 0.86 & 0.3414 & 0.54 & 0.2522 & 0.84 & 17.876 & -0.29 & 0.2186 & 0.82 \\
 & 172254 & 0.1071 & 1.35 & 0.2691 & 0.34 & 0.0961 & 1.39 & 13.369 & 0.42 & 0.0826 & 1.40 \\
\hline
M-T$^{1}$  & 928 & 0.2416 & -  & 0.3280 & -  & 0.2062 & -  & 0.5915 & -  & 0.1821 & -  \\
 & 6543 & 0.1304 & 0.89 & 0.2827 & 0.21 & 0.1142 & 0.85 & 0.4517 & 0.39 & 0.1042 & 0.81 \\
 & 49037 & 0.0509 & 1.36 & 0.2654 & 0.09 & 0.0395 & 1.53 & 0.2682 & 0.75 & 0.0351 & 1.57 \\
 & 379481 & 0.0192 & 1.41 & 0.2643 & 0.01 & 0.0111 & 1.83 & 0.1428 & 0.91 & 0.0098 & 1.84
\end{tabular}}
\caption{Results of relative error of cylindrical bending of a plate with $\nu=0.4999$.}
\label{tab:res_cyl_bend_plate_incompr}
\end{table}

\subsection{High-order bending}
\label{subsec:ho_bending}
We consider a high-order bending problem of a box with dimensions $L=H=t=10$ \cite{XHZH2019}, see Figure~\ref{fig:ho_bending}. The material parameters are as in Section~\ref{subsec:cylindrical_bending_plate}. The non-zero components of the exact solution are prescribed by

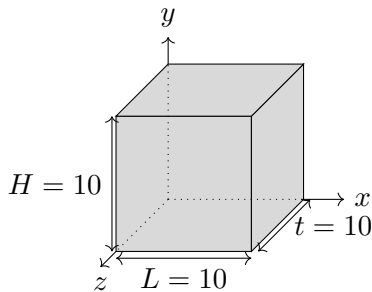
\begin{figure}[ht!]
	\centering
		\begin{tikzpicture}[scale=0.14]
			\def\L{10} 
			\def\H{10} 
			\def\t{10}

			\draw[fill=gray!30] (0,0,0) -- ++(\L,0,0) -- ++(0,\H,0) -- ++(-\L,0,0) -- cycle;
			\draw[fill=gray!30] (\L,0,0) -- ++(0,0,-\t) -- ++(0,\H,0) -- ++(0,0,\t) -- cycle;
			\draw[fill=gray!30] (0,\H,0) -- ++(\L,0,0) -- ++(0,0,-\t) -- ++(-\L,0,0) -- cycle;

			\draw[<->] (\L,-0.5,0) -- (0,-0.5,0) node[midway, below] {$L=10$};
			\draw[<->] (\L*1.05,0,0) -- (\L*1.05,0,-\t) node[midway, right] {$t=10$};
			\draw[<->] (-0.3,0,0) -- (-0.3,\H,0) node[midway, left] {$H=10$};

			\draw[-, dotted] (0,0,-\t) -- ++(\L,0,0);
			\draw[->] (\L,0,-\t)-- ++(3,0,0) node[right] {$x$};
			\draw[-, dotted] (0,0,-\t) -- ++(0,\H,0);
			\draw[->] (0,\H,-\t)-- ++(0,2,0) node[above] {$y$};
			\draw[-, dotted] (0,0,-\t) -- ++(0,0,\t);
			\draw[->] (0,0,0)-- ++(0,0,3)  node[below] {$z$};
	
		\end{tikzpicture}
  \vspace*{-0.5cm}
	\caption{Geometry of high-order bending of box example.}
	\label{fig:ho_bending}
\end{figure}

\begin{flalign*}
	&u_x \!=\! \frac{x\left(y-\frac{H}{2}\right)}{L},\, u_y \!=\! \tilde{u}_y-\frac{\int_{\Omega}\tilde{u}_y\,dx}{|\Omega|},\,  \tilde{u}_y \!=\! -\frac{1.5}{L}x^2-\frac{\lambda\,(2y-H)^2}{8L(\lambda+2\mu)}+\frac{2x^3}{3L^2},\,\omega_z \!=\! -\frac{x^2}{L^2},
\end{flalign*}
where $u_y$ has zero mean. The resulting stresses and body forces can be computed by inserting the exact solution into the strong form of Cosserat elasticity \eqref{eq:strong_cosserat}. For the boundary conditions, we fix $u_x$ at $x=0$ and $u_z$ at $z=0$ and $z=t$. Further, $\omega_x = 0$ at $z=0$ and $z=t$, $\omega_y=0$ at $x=0$, $z=0$, and $z=t$, as well as $\omega_z=0$ at $x=0$. A Lagrange multiplier enforces the zero mean of $u_y$. In Table~\ref{tab:res_ho_bend_p1} and Table~\ref{tab:res_ho_bend_p2}, we display the convergence rates for linear and quadratic elements, respectively. We note that cubic ansatz functions would resemble the exact solution. Again, due to the simple form of the exact solution, parts of the fields converge super-optimally. However, all three methods show the same convergence behavior.

\begin{table}
  \centering
  \resizebox{0.9\textwidth}{!}{
\begin{tabular}{c|c|cc|cc|cc|cc|cc}
& ncdof & $u_{\mathrm{err}}$ & eoc & $\omega_{\mathrm{err}}$ & eoc & $\tilde{\omega}_{\mathrm{err}}$ & eoc & $\sigma_{\mathrm{err}}$ & eoc & $\couplestress_{\mathrm{err}}$ & eoc \\
 \hline
$\Lag^1$  & 23 & 0.4974 & -  & 0.7595 & -   & - & -  & 0.7053 & -  & 0.8540 & -  \\
 & 136 & 0.2567 & 0.95 & 0.3625 & 1.07  & - & -  & 0.4649 & 0.60 & 0.3519 & 1.28 \\
 & 920 & 0.0910 & 1.50 & 0.1226 & 1.56  & - & -  & 0.2574 & 0.85 & 0.1395 & 1.34 \\
 & 6736 & 0.0312 & 1.55 & 0.0354 & 1.79  & - & -  & 0.1328 & 0.95 & 0.0596 & 1.23 \\
\hline
M$^{1}$  & 94 & 0.5185 & -  & 0.9089 & -  & 0.6098 & -  & 0.7068 & -  & 0.6001 & -  \\
 & 679 & 0.2493 & 1.06 & 0.5261 & 0.79 & 0.3003 & 1.02 & 0.4657 & 0.60 & 0.3407 & 0.82 \\
 & 5197 & 0.0871 & 1.52 & 0.3934 & 0.42 & 0.1071 & 1.49 & 0.2578 & 0.85 & 0.1576 & 1.11 \\
 & 40729 & 0.0303 & 1.52 & 0.3774 & 0.06 & 0.0344 & 1.64 & 0.1332 & 0.95 & 0.0720 & 1.13 \\
\hline
M-T$^{1}$  & 203 & 0.3082 & -  & 0.6615 & -  & 0.3262 & -  & 0.4655 & -  & 0.4890 & -  \\
 & 1485 & 0.1628 & 0.92 & 0.4540 & 0.54 & 0.1773 & 0.88 & 0.3292 & 0.50 & 0.2838 & 0.79 \\
 & 11417 & 0.0634 & 1.36 & 0.3860 & 0.23 & 0.0699 & 1.34 & 0.1854 & 0.83 & 0.1416 & 1.00 \\
 & 89649 & 0.0253 & 1.33 & 0.3776 & 0.03 & 0.0265 & 1.40 & 0.0970 & 0.93 & 0.0692 & 1.03
\end{tabular}}
\caption{Results of relative error for high-order bending of box for $k=1$.}
\label{tab:res_ho_bend_p1}
\end{table}

\begin{table}
  \centering
  \resizebox{0.9\textwidth}{!}{
\begin{tabular}{c|c|cc|cc|cc|cc|cc}
& ncdof & $u_{\mathrm{err}}$ & eoc & $\omega_{\mathrm{err}}$ & eoc & $\tilde{\omega}_{\mathrm{err}}$ & eoc & $\sigma_{\mathrm{err}}$ & eoc & $\couplestress_{\mathrm{err}}$ & eoc \\
 \hline
$\Lag^2$  & 136 & 0.0606 & -  & 0.0448 & -   & - & -  & 0.1800 & -  & 0.0808 & -  \\
 & 920 & 0.0096 & 2.66 & 0.0047 & 3.25  & - & -  & 0.0442 & 2.03 & 0.0143 & 2.50 \\
 & 6736 & 0.0021 & 2.16 & 0.0006 & 3.05  & - & -  & 0.0114 & 1.96 & 0.0021 & 2.76 \\
 & 51488 & 0.0005 & 2.04 & 0.0001 & 3.02  & - & -  & 0.0029 & 1.98 & 0.0003 & 2.92 \\
\hline
M$^{2}$  & 313 & 0.0582 & -  & 0.1493 & -  & 0.0320 & -  & 0.1862 & -  & 0.0967 & -  \\
 & 2293 & 0.0096 & 2.61 & 0.0594 & 1.33 & 0.0039 & 3.04 & 0.0458 & 2.02 & 0.0166 & 2.54 \\
 & 17593 & 0.0022 & 2.15 & 0.0286 & 1.05 & 0.0005 & 2.92 & 0.0118 & 1.96 & 0.0024 & 2.77 \\
 & 137905 & 0.0005 & 2.04 & 0.0143 & 1.00 & 0.0001 & 2.97 & 0.0030 & 1.98 & 0.0003 & 2.92 \\
\hline
M-T$^{2}$  & 547 & 0.0427 & -  & 0.1422 & -  & 0.0192 & -  & 0.1285 & -  & 0.0770 & -  \\
 & 4063 & 0.0084 & 2.35 & 0.0589 & 1.27 & 0.0031 & 2.64 & 0.0333 & 1.95 & 0.0144 & 2.42 \\
 & 31381 & 0.0020 & 2.09 & 0.0286 & 1.04 & 0.0004 & 2.91 & 0.0087 & 1.94 & 0.0021 & 2.77 \\
 & 246793 & 0.0005 & 2.03 & 0.0143 & 1.00 & 0.0001 & 2.95 & 0.0022 & 1.97 & 0.0003 & 2.92
\end{tabular}}
\caption{Results of relative error for high-order bending of box for $k=2$.}
\label{tab:res_ho_bend_p2}
\end{table}

\subsection{Torsion of a cylinder}
\label{subsec:torsion_cylinder}
A cylinder with length $L=1$ and radius $R=0.2$ is exposed to a torsion force $T$ at the top and bottom anti-symmetrically; see Figure~\ref{fig:torsion_cylinder}. We consider the material parameters\footnote{In \cite{GJ1975} the parameters $\tilde{\mu}$ and $\kappa$ are used, which are related by $\mu = \tilde{\mu}+\nicefrac{\kappa}{2}$, $\mu_c = 2\kappa$.} $\mu=15$, $\lambda=1$, $\mu_c=5$, and $\alpha=\beta=\gamma=0.5$. The exact solution derived in \cite{GJ1975} reads in cylindrical coordinates 
\begin{flalign*}
	&u_r= u_z = 0,\quad u_\theta = C_1\,rz,\quad \omega_r = -\frac{C_1\,r}{2}+C_2 I_1(p\,r), \quad\omega_z = C_1z,\quad\omega_\theta = 0,\\
	& e_r = (x^2+y^2)^{-\half}\begin{pmatrix}
		x&y&0
	\end{pmatrix}^T,\,\, e_\theta = (x^2+y^2)^{-\half}\begin{pmatrix}
		-y&x&0\end{pmatrix}^T,\,\, e_z = \begin{pmatrix}
		0&0&1
	\end{pmatrix}^T,
\end{flalign*}
where $I_n(\cdot)$ is the modified Bessel function of first kind of order $n$ and
\begin{flalign*}
	&C_1 = 2C_2\left(\frac{\alpha+\beta+\gamma}{\beta+\gamma}p\,I_0(p\,R)-\frac{I_1(p\,R)}{R}\right),\quad\quad p = \sqrt{\frac{4\mu_c}{\alpha+\beta+\gamma}},\\
	&C_2 \!=\! \frac{T}{2\pi R^2}\left(\left(\frac{2\mu}{\beta\!+\!\gamma}\frac{R^2}{4}\!+\!\frac{3}{2}\right)(\alpha\!+\!\beta\!+\!\gamma)p\,I_0(pR)\!-\!\left(\frac{2\mu}{\beta\!+\!\gamma}\frac{R^2}{4}\!+\!2\right)(\beta\!+\!\gamma)\frac{I_1(paR)}{R}\right)^{-1}\!.
\end{flalign*}

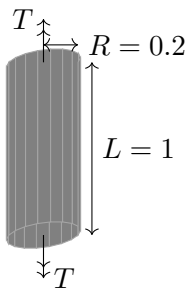
\begin{figure}[ht!]
	\centering
		\begin{tikzpicture}[scale=2.3]
			\def\R{0.2} 
			\def\L{1} 
			\def\nTheta{180} 
			\def\nZ{10} 
			\def\opacity{0.5} 
			\foreach \theta in {0,20,...,340} {
          \draw[gray!70,fill=gray, opacity=\opacity] ({\R*cos(\theta)}, 0, {\R*sin(\theta)}) -- ({ \R*cos(\theta+20)}, 0, {\R*sin(\theta+20)}) -- ({\R*cos(\theta+20)}, \L, {\R*sin(\theta+20)}) -- ({\R*cos(\theta)}, \L, {\R*sin(\theta)}) -- cycle;
          }

        \draw[<->] (1.4*\R,0,0) -- (1.4*\R,\L,0) node[midway, right] {$L=1$};
        \draw[<->] (0,1.1*\L,0) -- (\R,1.1*\L,0) node[right] {$R=0.2$};
        \draw[->] (0, \L,0) --(0,1.19*\L,0);
        \draw[->] (0, \L,0) --(0,1.25*\L,0) node[left] {$T$};
        \draw[->] (0, 0,0) --(0,-0.19*\L,0);
        \draw[->] (0, 0,0) --(0,-0.25*\L,0) node[right] {$T$};
			\end{tikzpicture}
   \vspace*{-0.5cm}
	\caption{Geometry of torsion of cylinder example.}
	\label{fig:torsion_cylinder}
\end{figure}
By symmetry in the $z$-direction, we consider half of the domain. All other boundaries are free except the top, where the torsion is applied. 
We enforce zero mean-value for the $x$ and $y$ displacement and rotation quantities by Lagrange multipliers (the exact solution also has zero mean). 
Note, that the resulting couple stress $\couplestress$ is symmetric independently of the choice of $\beta$ and $\gamma$ for this benchmark as only $\couplestress_{rr}$, $\couplestress_{\theta\theta}$, and $\couplestress_{zz}$ are non-zero \cite{GJ1975}. We use linear, quadratic, and cubic elements. To reduce the geometry approximation error due to the curved cylindrical domain, we curve the boundary with polynomial degree $k-1$. An isoperimetric approach would be to curve the boundary equally with the displacement and rotation order $k$ for the Lagrangian method $\Lag^k$. However, we use $\RT^{k-1}$ for the rotation field in the MCS methods, which does not contain all polynomials of degree $k$. Thus, curving would reduce the method's accuracy instead of increasing it because the polynomials are not preserved during the mapping. 
We observe  
comparable convergence behavior for all three methods; see Figure~\ref{fig:torsion_res}. The pure MCS method is superior by a constant compared to the MCS-TDNNS method. However, the same convergence rates are obtained. Due to the symmetries of the solution, some super-convergence is observed.

\begin{figure}[ht!]
	\centering
	\resizebox{0.32\textwidth}{!}{\begin{tikzpicture}
		\begin{loglogaxis}[
			legend style={at={(0,0)}, anchor=south west},
			xlabel={ncdof},
			ylabel={error},
			ymajorgrids=true,
			grid style=dotted,
			]
        \input{torsion_p1_err_u_w.txt}
		\end{loglogaxis}
		
		\node (B) at (4.7, 2.4) [] {$\mathcal{O}(h)$};
		\node (B) at (4.7, 0.8) [] {$\mathcal{O}(h^2)$};
\end{tikzpicture}}
	\resizebox{0.32\textwidth}{!}{\begin{tikzpicture}
		\begin{loglogaxis}[
			legend style={at={(0,0)}, anchor=south west},
			xlabel={ncdof},
			ylabel={error},
			ymajorgrids=true,
			grid style=dotted,
			]

        \input{torsion_p2_err_u_w.txt}
		\end{loglogaxis}
		
		\node (B) at (3.5, 4.2) [] {$\mathcal{O}(h^2)$};
\end{tikzpicture}}
\resizebox{0.32\textwidth}{!}{\begin{tikzpicture}
		\begin{loglogaxis}[
			legend style={at={(0,0)}, anchor=south west},
			xlabel={ncdof},
			ylabel={error},
			ymajorgrids=true,
			grid style=dotted,
			]

        \input{torsion_p3_err_u_w.txt}
		\end{loglogaxis}
		
		\node (B) at (6, 2.2) [] {$\mathcal{O}(h^3)$};
        \node (B) at (4.5, 0.8) [] {$\mathcal{O}(h^4)$};
\end{tikzpicture}}
\vspace*{-0.45cm}
\caption{Convergence rates for cylinder torsion with $k=1$ (left), $k=2$ (middle), and $k=3$ (right). For M$^k$ and M-T$^k$ $\tilde{\omega}_h$ is used instead of $\omega_h$.}
\label{fig:torsion_res}
\end{figure}
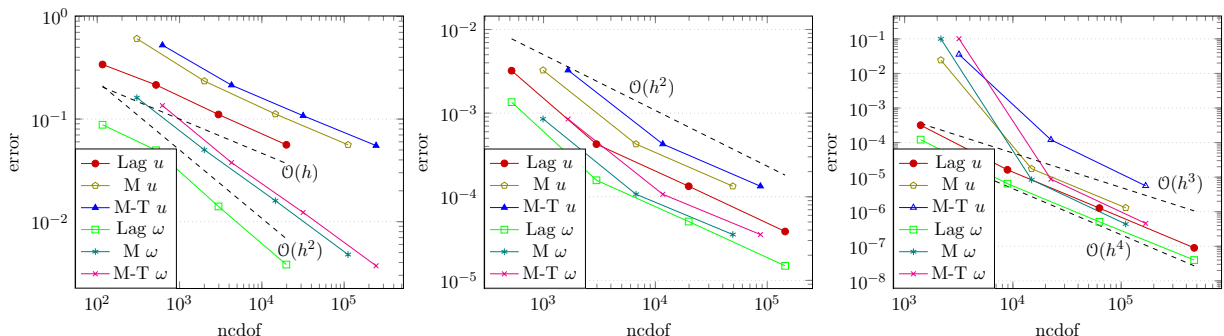

\subsection{Convergence and robustness in the Cosserat coupling constant}
\label{subsec:robustness_mu_c}
So far, the Cosserat coupling constant $\mu_c$ was well-behaved in the previous numerical examples. To test robustness in the Cosserat coupling constant $\mu_c\to\infty$, we adapt the example from Section~\ref{subsec:ho_bending}. We consider the domain $\Omega=(0,1)^3$ and the material parameters $E=2500$, $\nu = 0.25$ (and corresponding Lam\'e parameters $\mu$ and $\lambda$), $L_c=1$, $\alpha = \beta = 2\mu L_c^2$, and $\gamma = 4\mu L_c^2$. We take a sequence of Cosserat coupling moduli $\mu_c\in \{\mu,\mu\times 10^3,\mu\times 10^6\}$.  
The exact solution is given by
\begin{flalign*}
&u_x \!=\! \sin(x)(y-\half),\,\, u_y \!=\! -\half \sin(x)^2-\sin(x)^2\frac{\lambda}{2(\lambda+2\mu)}(y-\half)^2\cos(z)+\frac{\sin(x)^3}{3},\\
&u_z \!=\! \sin(x)^2\cos(1\!-\!y)(z\!-\!\half),\,\, \omega \!=\! \half\curl u \!+\! \frac{1}{\mu_c}\grad\Phi,\,\, \Phi \!=\! 10^3 x^2(1\!-\!x)y(1\!-\!y)(1\!-\!z)^2
\end{flalign*}
and the corresponding right-hand sides can be computed by inserting the exact solution into the strong form \eqref{eq:strong_cosserat}.
At the boundary $x=0$, we have homogeneous Dirichlet data for $u$ and $\omega$ (clamped boundary), the other boundaries are of Neumann type. 
\begin{table}
  \centering
  \resizebox{0.9\textwidth}{!}{
\begin{tabular}{c|c|cc|cc|cc|cc|cc}
& ncdof & $u_{\mathrm{err}}$ & eoc & $\omega_{\mathrm{err}}$ & eoc & $\tilde{\omega}_{\mathrm{err}}$ & eoc & $\sigma_{\mathrm{err}}$ & eoc & $\couplestress_{\mathrm{err}}$ & eoc \\ 
 \hline
$\Lag^1$  & 30 & 0.6248 & -  & 0.6333 & -   & - & -  & 0.5892 & -  & 0.7437 & -  \\
 & 156 & 0.3494 & 0.84 & 0.3337 & 0.92  & - & -  & 0.3642 & 0.69 & 0.4115 & 0.85 \\
 & 984 & 0.1846 & 0.92 & 0.1823 & 0.87  & - & -  & 0.1921 & 0.92 & 0.2269 & 0.86 \\
 & 6960 & 0.0936 & 0.98 & 0.0943 & 0.95  & - & -  & 0.0975 & 0.98 & 0.1169 & 0.96 \\
 & 52320 & 0.0470 & 0.99 & 0.0478 & 0.98  & - & -  & 0.0489 & 1.00 & 0.0590 & 0.99 \\
\hline
M$^1$  & 99 & 0.5843 & -  & 1.2808 & -  & 0.8426 & -  & 0.6296 & -  & 0.7276 & -  \\
 & 702 & 0.3444 & 0.76 & 1.2522 & 0.03 & 0.4509 & 0.90 & 0.3667 & 0.78 & 0.3913 & 0.89 \\
 & 5292 & 0.1840 & 0.90 & 1.2940 & -0.05 & 0.2442 & 0.88 & 0.1929 & 0.93 & 0.2107 & 0.89 \\
 & 41112 & 0.0936 & 0.98 & 1.3186 & -0.03 & 0.1221 & 1.00 & 0.0978 & 0.98 & 0.1063 & 0.99 \\
 & 324144 & 0.0470 & 0.99 & 1.3326 & -0.02 & 0.0608 & 1.00 & 0.0491 & 1.00 & 0.0533 & 1.00 \\
\hline
M-T$^1$  & 210 & 0.5654 & -  & 1.3104 & -  & 0.8669 & -  & 0.3566 & -  & 0.7402 & -  \\
 & 1524 & 0.3125 & 0.86 & 1.2584 & 0.06 & 0.4533 & 0.94 & 0.1925 & 0.89 & 0.3927 & 0.91 \\
 & 11592 & 0.1696 & 0.88 & 1.2957 & -0.04 & 0.2445 & 0.89 & 0.1065 & 0.85 & 0.2108 & 0.90 \\
 & 90384 & 0.0873 & 0.96 & 1.3191 & -0.03 & 0.1221 & 1.00 & 0.0554 & 0.94 & 0.1063 & 0.99 \\
 & 713760 & 0.0442 & 0.98 & 1.3327 & -0.01 & 0.0608 & 1.01 & 0.0282 & 0.98 & 0.0533 & 1.00
\end{tabular}}
\caption{Results robustness for Cosserat coupling constant $\mu_c=\mu$.}
\label{tab:res_robust_mu1e3}
\end{table}

\begin{table}
  \centering
  \resizebox{0.9\textwidth}{!}{
\begin{tabular}{c|c|cc|cc|cc|cc|cc}
& ncdof & $u_{\mathrm{err}}$ & eoc & $\omega_{\mathrm{err}}$ & eoc & $\tilde{\omega}_{\mathrm{err}}$ & eoc & $\sigma_{\mathrm{err}}$ & eoc & $\couplestress_{\mathrm{err}}$ & eoc \\ 
 \hline
$\Lag^1$  & 30 & 0.9805 & -  & 0.9650 & -   & - & -  & 0.9333 & -  & 0.9643 & -  \\
 & 156 & 0.7422 & 0.40 & 0.8471 & 0.19  & - & -  & 0.5656 & 0.72 & 0.8422 & 0.20 \\
 & 984 & 0.4997 & 0.57 & 0.5768 & 0.55  & - & -  & 0.4081 & 0.47 & 0.5625 & 0.58 \\
 & 6960 & 0.2331 & 1.10 & 0.2565 & 1.17  & - & -  & 0.2235 & 0.87 & 0.2452 & 1.20 \\
 & 52320 & 0.0910 & 1.36 & 0.0835 & 1.62  & - & -  & 0.1120 & 1.00 & 0.0818 & 1.58 \\
\hline
M$^1$  & 99 & 0.7587 & -  & 1.1637 & -  & 0.5918 & -  & 0.8166 & -  & 0.5447 & -  \\
 & 702 & 0.4261 & 0.83 & 1.1997 & -0.04 & 0.3082 & 0.94 & 0.5419 & 0.59 & 0.2934 & 0.89 \\
 & 5292 & 0.2407 & 0.82 & 1.2878 & -0.10 & 0.1539 & 1.00 & 0.3467 & 0.64 & 0.1480 & 0.99 \\
 & 41112 & 0.1260 & 0.93 & 1.3208 & -0.04 & 0.0770 & 1.00 & 0.1897 & 0.87 & 0.0737 & 1.01 \\
 & 324144 & 0.0642 & 0.97 & 1.3352 & -0.02 & 0.0406 & 0.92 & 0.0978 & 0.96 & 0.0381 & 0.95 \\
\hline
M-T$^1$  & 210 & 0.6232 & -  & 1.1915 & -  & 0.6402 & -  & 0.4062 & -  & 0.5663 & -  \\
 & 1524 & 0.3308 & 0.91 & 1.2223 & -0.04 & 0.3275 & 0.97 & 0.2174 & 0.90 & 0.2864 & 0.98 \\
 & 11592 & 0.1790 & 0.89 & 1.2929 & -0.08 & 0.1664 & 0.98 & 0.1195 & 0.86 & 0.1483 & 0.95 \\
 & 90384 & 0.0922 & 0.96 & 1.3213 & -0.03 & 0.0844 & 0.98 & 0.0622 & 0.94 & 0.0763 & 0.96 \\
 & 713760 & 0.0467 & 0.98 & 1.3350 & -0.01 & 0.0426 & 0.99 & 0.0317 & 0.97 & 0.0389 & 0.97
\end{tabular}}
\caption{Results robustness for Cosserat coupling constant $\mu_c=\mu \times 10^3$.}
\label{tab:res_robust_mu1e6}
\end{table}

\begin{table}
  \centering
  \resizebox{0.9\textwidth}{!}{
\begin{tabular}{c|c|cc|cc|cc|cc|cc}
& ncdof & $u_{\mathrm{err}}$ & eoc & $\omega_{\mathrm{err}}$ & eoc & $\tilde{\omega}_{\mathrm{err}}$ & eoc & $\sigma_{\mathrm{err}}$ & eoc & $\couplestress_{\mathrm{err}}$ & eoc \\ 
 \hline
$\Lag^1$  & 30 & 1.0071 & -  & 1.0000 & -   & - & -  & 0.9554 & -  & 1.0000 & -  \\
 & 156 & 0.8425 & 0.26 & 0.9998 & 0.00  & - & -  & 0.5816 & 0.72 & 0.9998 & 0.00 \\
 & 984 & 0.8061 & 0.06 & 0.9992 & 0.00  & - & -  & 0.4954 & 0.23 & 0.9991 & 0.00 \\
 & 6960 & 0.7854 & 0.04 & 0.9968 & 0.00  & - & -  & 0.4070 & 0.28 & 0.9965 & 0.00 \\
 & 52320 & 0.7725 & 0.02 & 0.9873 & 0.01  & - & -  & 0.3673 & 0.15 & 0.9862 & 0.01 \\
\hline
M$^1$  & 99 & 0.7612 & -  & 1.1649 & -  & 0.5931 & -  & 0.8193 & -  & 0.5458 & -  \\
 & 702 & 0.4285 & 0.83 & 1.2006 & -0.04 & 0.3104 & 0.93 & 0.5456 & 0.59 & 0.2968 & 0.88 \\
 & 5292 & 0.2432 & 0.82 & 1.2878 & -0.10 & 0.1581 & 0.97 & 0.3507 & 0.64 & 0.1531 & 0.95 \\
 & 41112 & 0.1289 & 0.92 & 1.3206 & -0.04 & 0.0817 & 0.95 & 0.1932 & 0.86 & 0.0796 & 0.94 \\
 & 324144 & 0.0692 & 0.90 & 1.3354 & -0.02 & 0.0450 & 0.86 & 0.1022 & 0.92 & 0.0434 & 0.87 \\
\hline
M-T$^1$  & 210 & 0.6224 & -  & 1.1899 & -  & 0.6388 & -  & 0.4058 & -  & 0.5663 & -  \\
 & 1524 & 0.3307 & 0.91 & 1.2208 & -0.04 & 0.3253 & 0.97 & 0.2174 & 0.90 & 0.2851 & 0.99 \\
 & 11592 & 0.1789 & 0.89 & 1.2919 & -0.08 & 0.1626 & 1.00 & 0.1196 & 0.86 & 0.1455 & 0.97 \\
 & 90384 & 0.0922 & 0.96 & 1.3209 & -0.03 & 0.0806 & 1.01 & 0.0622 & 0.94 & 0.0728 & 1.00 \\
 & 713760 & 0.0467 & 0.98 & 1.3349 & -0.02 & 0.0400 & 1.01 & 0.0317 & 0.97 & 0.0363 & 1.00
\end{tabular}}
\caption{Results robustness for Cosserat coupling constant $\mu_c=\mu \times 10^6$.}
\label{tab:res_robust_mu1e9}
\end{table}

In Table~\ref{tab:res_robust_mu1e3}, Table~\ref{tab:res_robust_mu1e6}, and Table~\ref{tab:res_robust_mu1e9} the results for the linear methods are displayed. For the case $\mu_c=\mu$ in Table~\ref{tab:res_robust_mu1e3}, all three methods perform equally well. No locking is observed. Increasing the Cosserat coupling constant to $\mu_c=\mu\times 10^3$, however, yields an increased so-called pre-asymptotic regime for the pure Lagrangian method $\Lag^1$. This pre-asymptotic regime can also be clearly observed in Figure~\ref{fig:res_robustness} (left), where the mesh resolution has to be sufficiently fine before optimal convergence starts. 
As expected by the convergence Theorem~\ref{thm:convergence} and Corollary~\ref{cor:stability_convergence_mcs}, where the constant $C$ has been shown to be independent of the Cosserat coupling constant $\mu_c$, 
the two mixed methods do not show this locking phenomenon, cf. Table~\ref{tab:res_robust_mu1e6}. 
Even by increasing the constant to $\mu_c=\mu\times 10^6$, which already almost resembles the coupled stress limit problem, the mixed methods converge optimally. In contrast, the Lagrangian-based method $\Lag^1$ nearly stops converging, the pre-asymptotic regime completely dominates; see Table~\ref{tab:res_robust_mu1e9} and Figure~\ref{fig:res_robustness} (left). Increasing the polynomial ansatz space from linear to quadratic elements reduces the enormous locking problem for $\Lag^2$. Nevertheless, one can observe in Figure~\ref{fig:res_robustness} (right) that the convergence is not independent of $\mu_c$; the curves shift strongly. The convergence for the quadratic mixed methods shows the same behavior as their linear counterparts.

\begin{figure}[ht!]
	\centering
	\resizebox{0.47\textwidth}{!}{\begin{tikzpicture}
		\begin{loglogaxis}[
			legend style={at={(1,0)}, anchor=south west},
			xlabel={ncdof},
			ylabel={$\|u_h-u\|_{V_h}/\|u\|_{H^1}$},
			ymajorgrids=true,
			grid style=dotted,
			]
        \input{robustness_p1_err_u_h1.txt}
		
		\end{loglogaxis}
\end{tikzpicture}}
\resizebox{0.47\textwidth}{!}{\begin{tikzpicture}
		\begin{loglogaxis}[
			legend style={at={(1,0)}, anchor=south west},
			xlabel={ncdof},
			ylabel={$\|u_h-u\|_{V_h}/\|u\|_{H^1}$},
			ymajorgrids=true,
			grid style=dotted,
			]
        \input{robustness_p2_err_u_h1.txt}
		
		\end{loglogaxis}
\end{tikzpicture}}
\vspace*{-0.45cm}
\caption{Results robustness test for $\mu_c/\mu\in \{1,10^3,10^6\}$ with methods of order $k=1,2$.}
\label{fig:res_robustness}
\end{figure}
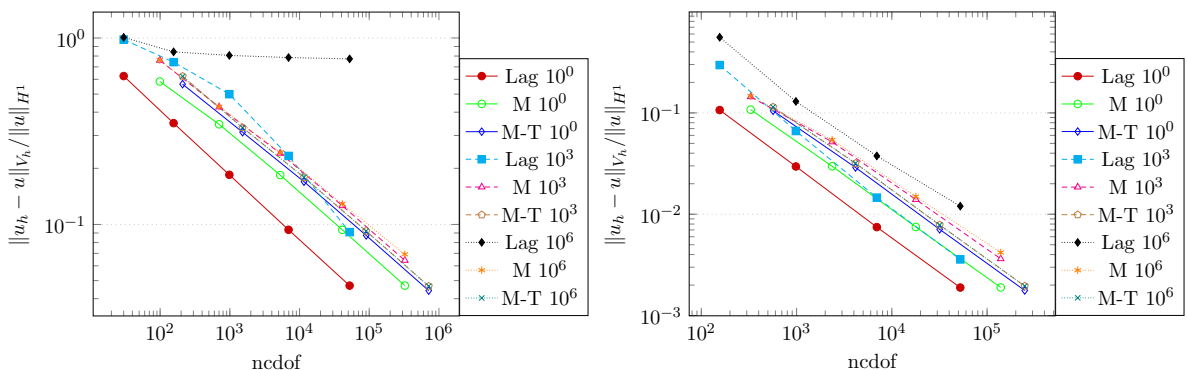

\section*{Acknowledgments}
The work of Kaibo Hu was supported by a Royal Society University Research Fellowship (URF$\backslash$R1$\backslash$221398). Michael Neunteufel acknowledges support by the National Science Foundation (USA) Grant DMS-219958.

\bibliographystyle{siam}
\bibliography{references}
\end{document}

%% file: torsion_p1_err_u_w.txt
\addlegendentry{Lag $u$}
\addplot[color=red, mark=*, style=solid]
coordinates {
	( 116, 0.3402982 )
	( 518, 0.21515661 )
	( 2972, 0.11081203 )
	( 19856, 0.05613719 )
};

\addlegendentry{M $u$}
\addplot[color=olive, mark=pentagon, style=solid]
coordinates {
	( 303, 0.60712699 )
	( 2001, 0.23461696 )
	( 14592, 0.11200227 )
	( 111498, 0.05613743 )
};

\addlegendentry{M-T $u$}
\addplot[color=blue, mark=triangle*, style=solid]
coordinates {
	( 620, 0.52510904 )
	( 4282, 0.21404816 )
	( 31752, 0.10781 )
	( 244332, 0.05511989 )
};

\addlegendentry{Lag $\omega$}
\addplot[color=green, mark=square, style=solid]
coordinates {
	( 116, 0.08800944 )
	( 518, 0.0496514 )
	( 2972, 0.01407156 )
	( 19856, 0.0038189 )
};

\addlegendentry{M $\omega$}
\addplot[color=teal, mark=asterisk, style=solid]
coordinates {
	( 303, 0.16184014 )
	( 2001, 0.05010548 )
	( 14592, 0.01600729 )
	( 111498, 0.00478074 )
};

\addlegendentry{M-T $\omega$}
\addplot[color=magenta, mark=x, style=solid]
coordinates {
	( 620, 0.13539725 )
	( 4282, 0.03767131 )
	( 31752, 0.01230013 )
	( 244332, 0.00371717 )
};

\addplot[color=black, dashed]
coordinates {
	( 116, 0.20504412815784637 )
	( 518, 0.12451549901402126 )
	( 2972, 0.06955319186184931 )
	( 19856, 0.036929158533791695 )
};

\addplot[color=black, dashed]
coordinates {
	( 116, 5*0.20504412815784637^2 )
	( 518, 5*0.12451549901402126^2 )
	( 2972, 5*0.06955319186184931^2 )
	( 19856, 5*0.036929158533791695^2 )
};

%% file: torsion_p2_err_u_w.txt
\addlegendentry{Lag $u$}
\addplot[color=red, mark=*, style=solid]
coordinates {
	( 518, 0.00322205 )
	( 2972, 0.00042573 )
	( 19856, 0.00013365 )
	( 144536, 3.845e-05 )
};

\addlegendentry{M $u$}
\addplot[color=olive, mark=pentagon, style=solid]
coordinates {
	( 990, 0.00326099 )
	( 6708, 0.00042682 )
	( 49242, 0.00013382 )
};

\addlegendentry{M-T $u$}
\addplot[color=blue, mark=triangle*, style=solid]
coordinates {
	( 1657, 0.00326299 )
	( 11641, 0.00042692 )
	( 86938, 0.00013382 )
};

\addlegendentry{Lag $\omega$}
\addplot[color=green, mark=square, style=solid]
coordinates {
	( 518, 0.00135404 )
	( 2972, 0.00015764 )
	( 19856, 5.056e-05 )
	( 144536, 1.493e-05 )
};

\addlegendentry{M $\omega$}
\addplot[color=teal, mark=asterisk, style=solid]
coordinates {
	( 990, 0.00085357 )
	( 6708, 0.00010723 )
	( 49242, 3.535e-05 )
};

\addlegendentry{M-T $\omega$}
\addplot[color=magenta, mark=x, style=solid]
coordinates {
	( 1657, 0.00085017 )
	( 11641, 0.0001071 )
	( 86938, 3.534e-05 )
};

\addplot[color=black, dashed]
coordinates {
	( 518, 0.015504109494710728/2 )
	( 2972, 0.004837646498171222/2 )
	( 19856, 0.00136376275001392/2 )
	( 144536, 0.00036309141372143527/2 )
};

%% file: torsion_p3_err_u_w.txt
\addlegendentry{Lag $u$}
\addplot[color=red, mark=*, style=solid]
coordinates {
	( 1406, 0.00031826 )
	( 8906, 1.63e-05 )
	( 62948, 1.27e-06 )
	( 472352, 9e-08 )
};

\addlegendentry{M $u$}
\addplot[color=olive, mark=pentagon, style=solid]
coordinates {
	( 2163, 0.02434184 )
	( 14895, 1.749e-05 )
	( 110100, 1.3e-06 )
};

\addlegendentry{M-T $u$}
\addplot[color=blue, mark=triangle, style=solid]
coordinates {
	( 3180, 0.03531617 )
	( 22480, 0.00012007 )
	( 168332, 5.6e-06 )
};

\addlegendentry{Lag $\omega$}
\addplot[color=green, mark=square, style=solid]
coordinates {
	( 1406, 0.00012156 )
	( 8906, 6.48e-06 )
	( 62948, 5.1e-07 )
	( 472352, 4e-08 )
};

\addlegendentry{M $\omega$}
\addplot[color=teal, mark=asterisk, style=solid]
coordinates {
	( 2163, 0.10006646 )
	( 14895, 8.38e-06 )
	( 110100, 4.4e-07 )
};

\addlegendentry{M-T $\omega$}
\addplot[color=magenta, mark=x, style=solid]
coordinates {
	( 3180, 0.10081584 )
	( 22480, 8.72e-06 )
	( 168332, 4.6e-07 )
};

\addplot[color=black, dashed]
coordinates {
	( 1406, 0.0007112375533428165/2 )
	( 8906, 0.00011228385358185493/2 )
	( 62948, 1.5886128232827094e-05/2 )
	( 472352, 2.11706523948242e-06/2 )
};

\addplot[color=black, dashed]
coordinates {
	( 1406, 0.0007112375533428165^1.33333 )
	( 8906, 0.00011228385358185493^1.33333 )
	( 62948, 1.5886128232827094e-05^1.33333 )
	( 472352, 2.11706523948242e-06^1.33333 )
};

%% file: robustness_p1_err_u_h1.txt
\addlegendentry{Lag $10^0$}
\addplot[color=red, mark=*, style=solid]
coordinates {
	( 30, 0.62480991 )
	( 156, 0.34942082 )
	( 984, 0.18463113 )
	( 6960, 0.09362795 )
	( 52320, 0.04699807 )
};

\addlegendentry{M $10^0$}
\addplot[color=green, mark=o, style=solid]
coordinates {
	( 99, 0.58433718 )
	( 702, 0.34435072 )
	( 5292, 0.18396696 )
	( 41112, 0.09357813 )
	( 324144, 0.04700641 )
};

\addlegendentry{M-T $10^0$}
\addplot[color=blue, mark=diamond, style=solid]
coordinates {
	( 210, 0.56539109 )
	( 1524, 0.31248584 )
	( 11592, 0.16963489 )
	( 90384, 0.0873199 )
	( 713760, 0.04418657 )
};

\addlegendentry{Lag $10^3$}
\addplot[color=cyan, mark=square*, style=densely dashed,mark options=solid]
coordinates {
	( 30, 0.98051147 )
	( 156, 0.74217848 )
	( 984, 0.49969013 )
	( 6960, 0.23305297 )
	( 52320, 0.09099885 )
};

\addlegendentry{M $10^3$}
\addplot[color=magenta, mark=triangle, style=densely dashed,mark options=solid]
coordinates {
	( 99, 0.75872824 )
	( 702, 0.42611551 )
	( 5292, 0.24071573 )
	( 41112, 0.12604295 )
	( 324144, 0.06416379 )
};

\addlegendentry{M-T $10^3$}
\addplot[color=brown, mark=pentagon, style=densely dashed,mark options=solid]
coordinates {
	( 210, 0.62322958 )
	( 1524, 0.33078554 )
	( 11592, 0.17897285 )
	( 90384, 0.09218031 )
	( 713760, 0.0466917 )
};

\addlegendentry{Lag $10^6$}
\addplot[color=black, mark=diamond*, style=densely dotted,mark options=solid]
coordinates {
	( 30, 1.00708515 )
	( 156, 0.84246938 )
	( 984, 0.80614421 )
	( 6960, 0.78541902 )
	( 52320, 0.77246149 )
};

\addlegendentry{M $10^6$}
\addplot[color=orange, mark=asterisk, style=densely dotted,mark options=solid]
coordinates {
	( 99, 0.76118336 )
	( 702, 0.42854408 )
	( 5292, 0.24315035 )
	( 41112, 0.1288936 )
	( 324144, 0.06918073 )
};

\addlegendentry{M-T $10^6$}
\addplot[color=teal, mark=x, style=densely dotted,mark options=solid]
coordinates {
	( 210, 0.62244802 )
	( 1524, 0.33069228 )
	( 11592, 0.17894871 )
	( 90384, 0.09219128 )
	( 713760, 0.04670358 )
};

%% file: robustness_p2_err_u_h1.txt
\addlegendentry{Lag $10^0$}
\addplot[color=red, mark=*, style=solid]
coordinates {
	( 156, 0.10661022 )
	( 984, 0.02959799 )
	( 6960, 0.00746267 )
	( 52320, 0.00189051 )
};

\addlegendentry{M $10^0$}
\addplot[color=green, mark=o, style=solid]
coordinates {
	( 330, 0.10763756 )
	( 2364, 0.0297372 )
	( 17880, 0.00748618 )
	( 139056, 0.00189474 )
};

\addlegendentry{M-T $10^0$}
\addplot[color=blue, mark=diamond, style=solid]
coordinates {
	( 567, 0.10497878 )
	( 4158, 0.0289261 )
	( 31788, 0.00709965 )
	( 248472, 0.00176906 )
};

\addlegendentry{Lag $10^3$}
\addplot[color=cyan, mark=square*, style=densely dashed,mark options=solid]
coordinates {
	( 156, 0.29640536 )
	( 984, 0.06626801 )
	( 6960, 0.01459671 )
	( 52320, 0.00358031 )
};

\addlegendentry{M $10^3$}
\addplot[color=magenta, mark=triangle, style=densely dashed,mark options=solid]
coordinates {
	( 330, 0.14408965 )
	( 2364, 0.0517035 )
	( 17880, 0.01397549 )
	( 139056, 0.00363396 )
};

\addlegendentry{M-T $10^3$}
\addplot[color=brown, mark=pentagon, style=densely dashed,mark options=solid]
coordinates {
	( 567, 0.11431692 )
	( 4158, 0.03181916 )
	( 31788, 0.00783192 )
	( 248472, 0.00195747 )
};

\addlegendentry{Lag $10^6$}
\addplot[color=black, mark=diamond*, style=densely dotted,mark options=solid]
coordinates {
	( 156, 0.55746922 )
	( 984, 0.12978318 )
	( 6960, 0.03761379 )
	( 52320, 0.01201623 )
};

\addlegendentry{M $10^6$}
\addplot[color=orange, mark=asterisk, style=densely dotted,mark options=solid]
coordinates {
	( 330, 0.14646943 )
	( 2364, 0.05416167 )
	( 17880, 0.01494298 )
	( 139056, 0.00419296 )
};

\addlegendentry{M-T $10^6$}
\addplot[color=teal, mark=x, style=densely dotted,mark options=solid]
coordinates {
	( 567, 0.11430936 )
	( 4158, 0.03182004 )
	( 31788, 0.00783459 )
	( 248472, 0.00195839 )
};

%% file: arxiv_intr_mfem_cosserat.bbl
\begin{thebibliography}{10}

\bibitem{CISMcosserat2010}
{\sc H.~Altenbach and V.~A. Eremeyev}, eds., {\em Generalized Continua from the Theory to Engineering Applications}, vol.~541 of CISM Courses and Lectures, Springer, 2010.

\bibitem{arnold2018finite}
{\sc D.~N. Arnold}, {\em Finite element exterior calculus}, SIAM, 2018.

\bibitem{Arnold.D;Falk.R;Winther.R.2006a}
{\sc D.~N. Arnold, R.~S. Falk, and R.~Winther}, {\em {Finite element exterior calculus, homological techniques, and applications}}, Acta numerica, 15 (2006), pp.~1--155.

\bibitem{arnold2007mixed}
{\sc D.~N. Arnold, R.~S. Falk, and R.~Winther}, {\em {Mixed finite element methods for linear elasticity with weakly imposed symmetry}}, Mathematics of Computation, 76 (2007), pp.~1699--1723.

\bibitem{Arnold.D;Falk.R;Winther.R.2010a}
{\sc D.~N. Arnold, R.~S. Falk, and R.~Winther}, {\em {Finite element exterior calculus: from Hodge theory to numerical stability}}, Bulletin of the American Mathematical Society, 47 (2010), pp.~281--354.

\bibitem{arnold2021complexes}
{\sc D.~N. Arnold and K.~Hu}, {\em Complexes from complexes}, Foundations of Computational Mathematics,  (2021), pp.~1739--1774.

\bibitem{boffi2013mixed}
{\sc D.~Boffi, F.~Brezzi, and M.~Fortin}, {\em {Mixed Finite Element Methods and Applications}}, Springer, 2013.

\bibitem{boon2024mixedfiniteelementmethods}
{\sc W.~M. Boon, O.~Duran, and J.~M. Nordbotten}, {\em Mixed finite element methods for linear cosserat equations}, SIAM Journal on Numerical Analysis, 63 (2025), pp.~306--333.

\bibitem{braess2008equilibrated}
{\sc D.~Braess and J.~Sch{\"o}berl}, {\em Equilibrated residual error estimator for edge elements}, Mathematics of Computation, 77 (2008), pp.~651--672.

\bibitem{brezzi1974existence}
{\sc F.~Brezzi}, {\em On the existence, uniqueness and approximation of saddle-point problems arising from {L}agrangian multipliers}, Revue fran{\c{c}}aise d'automatique, informatique, recherche op{\'e}rationnelle. Analyse num{\'e}rique, 8 (1974), pp.~129--151.

\bibitem{brezzi1985two}
{\sc F.~Brezzi, J.~Douglas~Jr, and L.~D. Marini}, {\em {Two families of mixed finite elements for second order elliptic problems}}, Numerische Mathematik, 47 (1985), pp.~217--235.

\bibitem{vcap2022bgg}
{\sc A.~{\v{C}}ap and K.~Hu}, {\em {BGG sequences with weak regularity and applications}}, Foundations of Computational Mathematics,  (2023), pp.~1--40.

\bibitem{CHZ2023}
{\sc L.~Chen, X.~Huang, and C.~Zhang}, {\em Distributional {{Finite Element}} curl div {{Complexes}} and {{Application}} to {{Quad Curl Problems}}}, 2023.

\bibitem{CLS2021}
{\sc J.-H. Choi, B.-C. Lee, and G.-D. Sim}, {\em A 10-node tetrahedral element with condensed {{Lagrange}} multipliers for the modified couple stress theory}, Computers \& Structures, 246 (2021), p.~106476.

\bibitem{christiansen2011linearization}
{\sc S.~H. Christiansen}, {\em On the linearization of {R}egge calculus}, Numerische Mathematik, 119 (2011), pp.~613--640.

\bibitem{christiansen2023extended}
{\sc S.~H. Christiansen, K.~Hu, and T.~Lin}, {\em {Extended Regge complex for linearized Riemann-Cartan geometry and cohomology}}, arXiv preprint arXiv:2312.11709,  (2023).

\bibitem{CGS2010}
{\sc B.~Cockburn, J.~Gopalakrishnan, and F.-J. Sayas}, {\em A projection-based error analysis of {{HDG}} methods}, Mathematics of Computation, 79 (2010), pp.~1351--1367.

\bibitem{Cosserat1909}
{\sc E.~Cosserat and F.~Cosserat}, {\em Th{\'{e}}ories des Corps D{\'{e}}formables}, Hermann, Paris, 1909.

\bibitem{GJ1975}
{\sc R.~D. Gauthier and W.~E. Jahsman}, {\em A {{Quest}} for {{Micropolar Elastic Constants}}}, Journal of Applied Mechanics, 42 (1975), pp.~369--374.

\bibitem{gopalakrishnan2020mass}
{\sc J.~Gopalakrishnan, P.~L. Lederer, and J.~Sch\"oberl}, {\em A mass conserving mixed stress formulation for {S}tokes flow with weakly imposed stress symmetry}, SIAM Journal on Numerical Analysis, 58 (2020), pp.~706--732.

\bibitem{gopalakrishnan2020mass0}
{\sc J.~Gopalakrishnan, P.~L. Lederer, and J.~Sch{\"o}berl}, {\em A mass conserving mixed stress formulation for the {S}tokes equations}, IMA Journal of Numerical Analysis, 40 (2020), pp.~1838--1874.

\bibitem{grbvcic2018variational}
{\sc S.~Grb{\v{c}}i{\'c}, A.~Ibrahimbegovi{\'c}, and G.~Jeleni{\'c}}, {\em {Variational formulation of micropolar elasticity using 3D hexahedral finite-element interpolation with incompatible modes}}, Computers \& Structures, 205 (2018), pp.~1--14.

\bibitem{Gunther1958de}
{\sc W.~G{\"u}nther}, {\em {Zur Statik und Kinematik des Cosseratschen Kontinuums}}, Abhandlungen der Braunschweigischen Wissenschaftlichen Gesellschaft,  (1958).

\bibitem{Hel67}
{\sc K.~Hellan}, {\em Analysis of elastic plates in flexure by a simplified finite element method}, Acta Polytechnica Scandinavica, Civil Engineering Series, 46 (1967).

\bibitem{Her67}
{\sc L.~R. Herrmann}, {\em Finite element bending analysis for plates}, Journal of the Engineering Mechanics Division, 93 (1967), pp.~13--26.

\bibitem{hu2023distributional}
{\sc K.~Hu, T.~Lin, and Q.~Zhang}, {\em Distributional {Hessian} and divdiv complexes on triangulation and cohomology}, SIAM Journal on Applied Algebra and Geometry, 9 (2025), pp.~108--153.

\bibitem{JN2010}
{\sc J.~Jeong and P.~Neff}, {\em Existence, {{Uniqueness}} and {{Stability}} in {{Linear Cosserat Elasticity}} for {{Weakest Curvature Conditions}}}, Mathematics and Mechanics of Solids, 15 (2010), pp.~78--95.

\bibitem{Joh73}
{\sc C.~Johnson}, {\em On the convergence of a mixed finite element method for plate bending moments}, Numerische Mathematik, 21 (1973), pp.~43--62.

\bibitem{Lakes2023}
{\sc R.~S. Lakes}, {\em Experimental evaluation of micromorphic elastic constants in foams and lattices}, Zeitschrift f\"ur angewandte Mathematik und Physik (ZAMP), 74 (2023).

\bibitem{AMM2010}
{\sc G.~Maugin and A.~Metrikine}, eds., {\em Mechanics of Generalized Continua – One Hundred Years After the Cosserats}, vol.~21 of Advances in Mechanics and Mathematics, Springer, 2010.

\bibitem{Nedelec.J.1986a}
{\sc J.~N\'{e}d\'{e}lec}, {\em {A new family of mixed finite elements in $\mathbb{R}^{3}$}}, Numerische Mathematik, 50 (1986), pp.~57--81.

\bibitem{NJ2009}
{\sc P.~Neff and J.~Jeong}, {\em A new paradigm: The linear isotropic {{Cosserat}} model with conformally invariant curvature energy}, Journal of Applied Mathematics and Mechanics/Zeitschrift f{\"u}r Angewandte Mathematik und Mechanik (ZAMM), 89 (2009), p.~107.

\bibitem{NJF2010}
{\sc P.~Neff, J.~Jeong, and A.~Fischle}, {\em Stable identification of linear isotropic {{Cosserat}} parameters: Bounded stiffness in bending and torsion implies conformal invariance of curvature}, Acta Mechanica, 211 (2010), pp.~237--249.

\bibitem{Neun21}
{\sc M.~Neunteufel}, {\em Mixed Finite Element Methods for Nonlinear Continuum Mechanics and Shells}, PhD thesis, TU Wien, 2021.

\bibitem{NPS2021}
{\sc M.~Neunteufel, A.~S. Pechstein, and J.~Sch{\"o}berl}, {\em Three-field mixed finite element methods for nonlinear elasticity}, Computer Methods in Applied Mechanics and Engineering, 382 (2021), p.~113857.

\bibitem{NS2019}
{\sc M.~Neunteufel and J.~Sch{\"o}berl}, {\em The {{Hellan}}--{{Herrmann}}--{{Johnson}} method for nonlinear shells}, Computers \& Structures, 225 (2019), p.~106109.

\bibitem{NS2024}
\leavevmode\vrule height 2pt depth -1.6pt width 23pt, {\em The {{Hellan}}--{{Herrmann}}--{{Johnson}} and {{TDNNS}} methods for linear and nonlinear shells}, Computers \& Structures, 305 (2024), p.~107543.

\bibitem{PG2008}
{\sc S.~K. Park and X.-L. Gao}, {\em Variational formulation of a modified couple stress theory and its application to a simple shear problem}, Zeitschrift f{\"u}r angewandte Mathematik und Physik, 59 (2008), pp.~904--917.

\bibitem{pechstein2011tangential}
{\sc A.~Pechstein and J.~Sch{\"o}berl}, {\em Tangential-displacement and normal--normal-stress continuous mixed finite elements for elasticity}, Mathematical Models and Methods in Applied Sciences, 21 (2011), pp.~1761--1782.

\bibitem{PS12}
{\sc A.~Pechstein and J.~Sch{\"o}berl}, {\em Anisotropic mixed finite elements for elasticity}, International Journal for Numerical Methods in Engineering, 90 (2012), pp.~196--217.

\bibitem{PS17}
{\sc A.~Pechstein and J.~Sch{\"o}berl}, {\em The {TDNNS} method for {R}eissner--{M}indlin plates}, Numerische Mathematik, 137 (2017), pp.~713--740.

\bibitem{providas2002finite}
{\sc E.~Providas and M.~Kattis}, {\em {Finite element method in plane Cosserat elasticity}}, Computers \& Structures, 80 (2002), pp.~2059--2069.

\bibitem{Raviart.P;Thomas.J.1977a}
{\sc P.~Raviart and J.~Thomas}, {\em A mixed finite element method for second order elliptic problems}, Lecture Notes in Mathematics, 606 (1977), pp.~292--315.

\bibitem{regge1961general}
{\sc T.~Regge}, {\em {General relativity without coordinates}}, Il Nuovo Cimento (1955-1965), 19 (1961), pp.~558--571.

\bibitem{riahi2009full}
{\sc A.~Riahi and J.~H. Curran}, {\em {Full 3D finite element Cosserat formulation with application in layered structures}}, Applied mathematical modelling, 33 (2009), pp.~3450--3464.

\bibitem{RuegerLakes2018}
{\sc Z.~Rueger and R.~S. Lakes}, {\em Strong {Cosserat} elasticity in a transversely isotropic polymer lattice}, Physical Review Letters,  (2018).

\bibitem{sander2010geodesic}
{\sc O.~Sander}, {\em {Geodesic finite elements for Cosserat rods}}, International journal for numerical methods in engineering, 82 (2010), pp.~1645--1670.

\bibitem{Sch97}
{\sc J.~Sch{\"o}berl}, {\em {NETGEN} an advancing front 2{D}/3{D}-mesh generator based on abstract rules}, Computing and Visualization in Science, 1 (1997), pp.~41--52.

\bibitem{Sch14}
\leavevmode\vrule height 2pt depth -1.6pt width 23pt, {\em C++ 11 implementation of finite elements in {NGS}olve}, Institute for Analysis and Scientific Computing, Vienna University of Technology,  (2014).

\bibitem{Sin2009}
{\sc A.~Sinwel}, {\em A New Family of Mixed Finite Elements for Elasticity}, PhD thesis, Johannes Kepler Universit{\"a}t Linz, {Linz}, 2009.

\bibitem{Ste1991}
{\sc R.~Stenberg}, {\em Postprocessing schemes for some mixed finite elements}, ESAIM: Mathematical Modelling and Numerical Analysis, 25 (1991), pp.~151--167.

\bibitem{XHZH2019}
{\sc Q.~Xie, Y.~Hu, Y.~Zhou, and W.~Han}, {\em Improving the bending response of four-node quadrilateral and eight-node hexahedral elements for {{Cosserat}} elasticity problems}, Engineering Computations, 36 (2019), pp.~1950--1976.

\end{thebibliography}
